\documentclass{amsart}
\usepackage[dvipsnames]{xcolor}
\definecolor{DarkPurple}{RGB}{88, 41, 123}
\usepackage[
    colorlinks=true,
    linkcolor=OliveGreen,   
    citecolor=DarkPurple, 
    urlcolor=NavyBlue
]{hyperref}
\usepackage{amsmath}
\usepackage{amsthm}
\usepackage{amssymb}
\usepackage{tikz-cd}
\usepackage{etoolbox}
\usepackage{enumitem}
\usepackage{xparse}

\usepackage{tikz}
\usetikzlibrary{calc}

\makeatletter
\newcommand*\leftdash{\rotatebox[origin=c]{-45}{$\dabar@\dabar@\dabar@$}}
\newcommand*\rightdash{\rotatebox[origin=c]{45}{$\dabar@\dabar@\dabar@$}}
\makeatother

\newcommand{\wqw}[3]{{#1}\leftdash{#2}\rightdash{#3}}

\newcommand{\qw}[2]{{#1}\rightdash{#2}}

\newtheorem{thm}{Theorem}[section]
\newtheorem{theorem}[thm]{Theorem}
\newtheorem{corollary}[thm]{Corollary}
\newtheorem{observation}[thm]{Observation}

\newtheorem{lemma}[thm]{Lemma}
\newtheorem{prop}[thm]{Proposition}
\newtheorem{proposition}[thm]{Proposition}

\newtheorem{conj}[thm]{Conjecture}

\newtheorem{claim}[thm]{Claim}

\theoremstyle{definition}
\newtheorem{definition}[thm]{Definition}
\newtheorem{example}[thm]{Example}
\newtheorem{remark}[thm]{Remark}

\newcommand{\gmod}{\mathfrak{g}\mathrm{-mod}}

\newcommand{\Afmod}{\mathcal{A}_f}
\newcommand{\KL}{\mathrm{KL}_{\mathcal{O}}}
\newcommand{\BGGO}{\mathcal{O}_{\hat{0}}}
\renewcommand{\O}{\mathcal{O}_{0}}
\newcommand{\tO}{\tilde{\mathcal{O}}_{0}}

\newcommand{\Da}{\mathcal{D}^{\mathrm{alg}}}
\newcommand{\Dg}{\mathcal{D}^{\mathrm{geom}}}

\title[Weight modules and gluing of sheaves on the flag variety]{Weight modules and \\gluing of sheaves on the flag variety}
\author{Pablo Boixeda Alvarez and Calder Morton-Ferguson}
\date{\today}

\begin{document}

\begin{abstract}
    We study a natural enlargement of the BGG Category $\mathcal{O}$ for a semisimple Lie algebra: the category of weight modules with trivial central character and finite-dimensional weight spaces supported on the root lattice. We give a geometric realization of this category as unipotently monodromic sheaves on the flag variety satisfying a singular support condition. We then explain that a derived version $\mathcal{D}$ of this category is Koszul dual to the Kazhdan--Laumon Category $\mathcal{O}$, a different enlargement of the BGG Category $\mathcal{O}$ obtained from a gluing construction for sheaves on the flag variety. This characterizes $\mathcal{D}$ as the category of algebras over a natural monad on a direct sum of copies of the derived Category $\mathcal{O}$. These results give new interpretations of classical algebraic constructions of weight modules over semisimple Lie algebras due to Fernando and Mathieu. We also conjecture that our results fit naturally into a proposed Koszul duality relating the small quantum group and the semi-infinite flag variety to the geometry of affine Springer fibers.
\end{abstract}

\maketitle

\section{Introduction}

The BGG Category $\mathcal{O}$ for a semisimple Lie algebra $\mathfrak{g}$ is one of the central objects of study in Lie theory, and its rich structure has been continually studied from algebraic and geometric perspectives over the past 50 years since its definition in \cite{BGG}. The principal goal of the present paper is to compare two natural enlargements of Category $\mathcal{O}$ which arise from gluing together a collection of copies of $\mathcal{O}$ indexed by the Weyl group $W$.

The idea to compare these two abelian categories arose from the observation that they have the same number of simple objects. Although upon first reflection it is tempting to conjecture an equivalence between these categories, in the present paper, we will explain that instead, these categories are \emph{Koszul dual} to one another. To make this claim precise, we will establish a Koszul duality equivalence between derived versions of these categories. 

We expect the approach in the present paper to have interesting applications and connection to other work. Most directly, we will explain that our results allow for a geometric interpretation of some of the classical constructions for weight modules, studied purely algebraically in \cite{Fernando} and \cite{Mathieu}, and we expect that further results in this direction, including a geometric reframing of the classification theorems in loc.\ cit., may follow from this approach. Further, we will explain the relationship between the Koszul duality illustrated in the present paper and the Koszul duality proposed in \cite{BBAMY} between modules over the small quantum group (alternatively, perverse sheaves on the semi-infinite flag variety) and microlocal sheaves on certain affine Springer fibers.

\subsection{Two categories}

We now introduce these two categories. The first, which we call the category of \emph{finite weight modules}, proceeds from the following basic procedure. In Definition \ref{def:cato}, we will recall that the principal block $\mathcal{O}_0$ of Category $\mathcal{O}$ is defined by a semisimplicity condition on the action of the Cartan subalgebra $\mathfrak{h}$ as well as a local finiteness condition on the nilpotent radical $\mathfrak{n}$ of a Borel subalgebra $\mathfrak{b}$ which contains $\mathfrak{h}$. (The precise definition which we will use is that in Definition \ref{def:catovariant}.) This definition therefore involves a choice of one of the $W$-many Borel subalgebras containing $\mathfrak{h}$, which we label as $\mathfrak{b}^w$ for $w \in W$. We let $\mathcal{O}_0^w$, for $w \in W$, be the resulting $W$-many versions of Category $\mathcal{O}$. We can then consider the smallest extension-closed subcategory containing all of these versions at once.
\begin{definition}
    The category $\Afmod$ of \emph{finite weight modules} is the Serre subcategory of $\gmod$ generated under extensions by
    \[\bigcup_{w \in W} \mathcal{O}_0^w \subset \gmod.\]
\end{definition}

The name we give $\Afmod$ in this definition is justified by the following result, which we prove in Corollary \ref{cor:afdefs}, which gives an even more natural description of $\Afmod$ in terms of some basic axioms similar to those used in the definition of Category $\mathcal{O}$ but now independent of any choice of Borel.

\begin{proposition}\label{prop:introafdef}
    The category $\Afmod$ is the full subcategory of modules $M \in \gmod$ satisfying the following conditions.
    \begin{enumerate}
        \item $M$ is a finitely-generated $\mathcal{U}(\mathfrak{g})$-module.
        \item The center $Z(\mathfrak{g}) \subset \mathcal{U}(\mathfrak{g})$ acts trivially on $M$.
        \item $M$ has a decomposition $$M = \bigoplus_{\lambda \in Q} M_\lambda$$
        on which $\mathfrak{h}$ acts locally finitely with generalized eigenvalue $\lambda$ on $M_\lambda$.
        \item Each weight space $M_\lambda$ is finite-dimensional.
    \end{enumerate}
\end{proposition}

The second category to be considered, which we compare with $\Afmod$ in the present paper, is the \emph{Kazhdan--Laumon Category $\mathcal{O}$}. In 1988, Kazhdan and Laumon introduced a formalism for gluing a collection of abelian categories equipped with functors called gluing functors. Their main example involved gluing $W$-many copies of the category of perverse sheaves on the basic affine space $G/U$ associated to a semisimple algebraic group $G$ along gluing functors called symplectic Fourier transforms. In \cite{KLCatO}, this gluing construction was studied after restricting to the subcategory $\mathrm{Perv}_B(G/U) \subset \mathrm{Perv}(G/U)$, which (as we explain in Section \ref{sec:bb}) can be viewed as a geometric model for $\mathcal{O}_0$. It it explained in loc.\ cit. (and was established even earlier, in \cite{P}) that after restricting to this category, the gluing endofunctors of $\mathrm{Perv}_B(G/U)$ which appear in Kazhdan--Laumon's constructed are given by convolution with certain standard and costandard sheaves (these convolutions are well-studied and are sometimes referred to as \emph{intertwining functors} in the literature). For the sake of the introduction only, we denote the corresponding endofunctors of $\mathcal{O}_0$ by $I_w$ for $w \in W$; they are explained more in Section \ref{sec:klgluing}. The result is the following category.

\begin{definition}
    The \emph{Kazhdan--Laumon category $\mathcal{O}$}, which we call $\mathrm{KL}_{\mathcal{O}}$ is the category of tuples $(A_w)_{w \in W}$ for $A_w \in \mathcal{O}_0$ equipped with morphisms
    \begin{align*}
        I_y(A_w) \to A_{yw}
    \end{align*}
    for each $y \in W$ satisfying the compatibilies in Definition \ref{def:gluedcat}.
\end{definition}

The starting point for why one might expect a meaningful comparison between these categories is a purely combinatorial observation. In the case $\mathfrak{g} = \mathfrak{sl}_2$, the category $\mathcal{O}_0$ (which depends on the choice of a Borel subalgebra containing $\mathfrak{h}$ which here we call $\mathfrak{b}$) has two simple objects: $L_{-2}^{\mathfrak{b}}$ (the Verma module with highest weight $-2$, which can be written as $\mathrm{Ind}_{\mathfrak{b}}^{\mathfrak{g}}\mathbb{C}_{-2}$) and $L_0$ (the trivial representation). One can show that there are three simple objects in $\Afmod$:
\begin{align*}
    L_{-2}^{\mathfrak{b}}, \quad L_0, \quad L_{2}^{\mathfrak{b}_-},
\end{align*}
where $L_{2}^{\mathfrak{b}_-} = \mathrm{Ind}_{\mathfrak{b}_-}^{\mathfrak{g}} \mathbb{C}_2$ for $\mathfrak{b}_-$ the opposite Borel subalgebra. By comparison, it is explained in \cite{KLCatO} that when $\mathfrak{g} = \mathfrak{sl}_2$ case there are three simple objects in $\KL$, given as tuples by
\begin{align*}
    (L_{0}, 0), \quad (L_{-2}^{\mathfrak{b}}, L_{-2}^{\mathfrak{b}}), \quad (0, L_0).
\end{align*}
A similar count for $\mathcal{A}_f$ when $\mathfrak{g} = \mathfrak{sl}_3$ yields $19$ simple objects; similarly, $19$ simple objects for $\mathrm{KL}_{\mathcal{O}}$ are enumerated explicitly in \cite[Figure 1]{KLCatO}. This observation was first made informally by Ben Webster.
\begin{observation}[Webster, 2023]
    The number of simple objects in $\Afmod$ is equal to the number of simple objects in $\mathrm{KL}_{\mathcal{O}}$ when $\mathfrak{g} = \mathfrak{sl}_3$.
\end{observation}

The present paper shows that this is not a coincidence, and the number of simple objects in these two categories is equal in any type. It was shown in \cite{KLCatO} that the number of simple objects in $\mathrm{KL}_{\mathcal{O}}$ is
    \begin{equation*}
\sum_{w \in W} |W/P(w)|
    \end{equation*}
where for any element $w\in W$, $P(w)$ is the standard parabolic subgroup of $W$ generated by all simple reflections $s$ with $\ell(sw) > \ell(w)$. We show in Corollary \ref{cor:countsimples} with a straightforward combinatorial proof that the same is true of $\Afmod$; this fact will also follow from the main result which we now explain.

In $\O$, one often considers \emph{tilting} objects; this definition can be adapted to provide a notion of tilting objects in $\KL$ as well. In \cite{BGS} (see also \cite{BY}), the \emph{Koszul duality} functor $\mathbb{K}$, an idempotent autoequivalence of $D^b(\O^{\mathrm{gr}})$, is introduced (where $\O^{\mathrm{gr}}$ is a technical variant of $\O$ consisting of objects of $\O$ equipped with a grading); this functor (following the normalizations in \cite{BY}) interchanges simple objects and tilting objects. In this paper we will explain examples where it is natural to give an explicit matching between the simple objects of $\Afmod$ and tilting objects of $\KL$, whereas attempting to directly give a bijection between the simple objects of these categories is much less natural.

This suggests that $\Afmod$ and $\KL$ are \emph{Koszul dual} rather than equivalent as abelian categories, and a precise formulation and proof of this fact comprises the main result of the present paper. Since Koszul duality takes place on the level of derived categories, we introduce in the present paper an appropriate derived version of $\Afmod$ which we call $\Da$ in addition to a category $\KL^\vee$ where the gluing functors on $D^b(\O)$ are twisted by $\mathbb{K}$, and then prove the following main result.
\begin{theorem}\label{thm:mainalg}
    There is an equivalence of dg categories
    \begin{align*}
        \Da \cong \KL^\vee
    \end{align*}
    between the derived version $\Da \subset D(\gmod)$ of the category $\Afmod$ of finite weight modules, and the \emph{Koszul dual Kazhdan--Laumon category $\O$} denoted $\KL^\vee$.
\end{theorem}

Although thus far, we have spoken of $\Afmod$ purely algebraically as a subcategory of $\gmod$, many of the objects, categories, and functors we consider can be richly understood geometrically via the Beilinson--Bernstein localization equivalence. This equivalence, between $D$-modules on the flag variety $G/B$ and the category $\gmod_0$ of $\mathfrak{g}$-modules with trivial central character, will be used throughout the paper to geometrically recast some of the main algebraic objects introduced thus far. The geometric version of $\Afmod$ will be the category of $T$-monodromic perverse sheaves on $G/B$ with unipotent monodromy and an interesting singular support condition. The derived version of this category $\Dg$ is equivalent to $\Da$ under the localization equivalence, and we will obtain the following geometric version of Theorem \ref{thm:mainalg}.
\begin{theorem}\label{main:geom}
    There is an equivalence of dg categories
    \begin{align*}
        \Dg \cong \mathrm{KL}_{\mathcal{O}}^\vee,
    \end{align*}
    where $\Dg$ is a certain completion of the category of unipotently $T$-monodromic sheaves on $G/B$ with singular support lying in
    \begin{align*}
        Z & = \bigcup_{w, z} T_{Y_w^z}^*(G/B),
    \end{align*}
    where for $w, z \in W$, $Y_w^z$ is the $z$-translate of the Schubert variety $Y_w = \overline{BwB} \subset G/B$ under left multiplication, so $Y_w^z = \overline{zBwB}$.
\end{theorem}
This geometric result allows us to study all of the objects considered up to this point using geometric tools. The category $\Dg$ can then be understood in two different ways: firstly, as the image under the localization functor of the category $\Da$ which arises naturally from finite weight modules as described. Secondly, we note that $\Dg$ and the Koszul duality perspective explored above can be understood in terms of microlocal sheaves on affine Springer fibers as follows.

\subsection{Connection to microlocal sheaves on affine Springer fibers}

We now explain a connection with the geometry of affine Springer fibers and the work in \cite{BBAMY}. In light of the Koszul duality to be described in \cite{BBAMY}, the equivalence in Theorem \ref{main:geom} should be understood as the Koszul dual of a functor from $\mathrm{KL}_{\mathcal{O}}$ to the semi-infinite flag variety constructed in \cite{KLCatO}. This is part of the original motivation of the present work, but is not directly necessary for our results, so we will leave out most of the details.

Denote by $\mathcal{F}\ell$ the affine flag variety of $G$ and write $P_-=\ker(G[t^{-1}]\rightarrow G)$. This group acts on $\mathcal{F}\ell$ such that there is an open inclusion $j : G/B\hookrightarrow P_-\backslash\mathcal{F}\ell$. For a regular semisimple $s\in\mathfrak{g}$ we can construct a character on $P_-$ using the group homomorphism $P_-\rightarrow\mathfrak{g}$ and pairing with $s$ under the Killing form; we denote this character by $\psi:P_-\rightarrow\mathbb{G}_a$.

In \cite{BBAMY}, the shifted Hamiltonian reduction $P_-\backslash\backslash\backslash_\psi\mathcal{F}\ell=:\mathcal{M}_\psi$ is constructed and a conical Lagrangian homeomorphic to the affine Springer fiber $\mathcal{F}\ell_{ts}$ is identified. The category $D_{(P_-,\psi)}(\mathcal{F}\ell)$ is then considered and a microlocalization functor
$$D_{(P_-,\psi)}(\mathcal{F}\ell)\rightarrow\mu \mathrm{Sh}(\mathcal{M}_\psi)$$
is defined. In particular we can consider the singular support of an object in $D_{(P_-,\psi)}(\mathcal{F}\ell)$ as a Lagrangian in $\mathcal{M}_\psi$. We denote by $\widetilde{\mathcal{D}}_\psi$ the full subcategory of $D_{(P_-,\psi)}(\mathcal{F}\ell)$ of objects whose singular support is contained in the Lagrangian $\mathcal{F}\ell_{ts}$ and further denote by $\mathcal{D}_\psi$ the full subcategory of objects whose singular support has finitely many components in $\mathcal{F}\ell_{ts}$. 

Denote by $\mathfrak{u}_q$ the small quantum group of the Langlands dual group $G^\vee$. We can consider the $\mathfrak{u}_q$-modules with a compatible grading by the character lattice and we will denote the category of these by $\overset{\bullet}{\mathfrak{u}} _q\mathrm{-mod}$. Denote by $\overset{\bullet}{\mathfrak{u}} _q\mathrm{-mod}_{\hat{0}}$ the principal block of this category.

The main result in \cite{BBAMY} is the following.
\begin{theorem}
    The category $\mathcal{D}_\psi$ is Koszul dual to the category of perfect complexes in $\overset{\bullet}{\mathfrak{u}} _q\mathrm{-mod}_{\hat{0}}$.
\end{theorem}

Note that we have an open embedding $j:G/B\hookrightarrow P_-\backslash \mathcal{F}\ell$. It follows that we have an open embedding $T^*(G/B)\hookrightarrow \mathcal{M}_\psi$. Denote by $\Lambda=\bigcup_{z} T_{Y_1^z}^*(G/B)\subset T^*(G/B)$. In \cite{LusRemarks} a locally closed subvariety of $\mathcal{F}\ell_{ts}$ is identified with $\Lambda$ and it is compatible with the diagram

\[\begin{tikzcd}
	\Lambda & {\mathcal{F}\ell_{ts}} \\
	{T^*(G/B)} & {\mathcal{M}_\psi}
	\arrow[hook, from=1-1, to=1-2]
	\arrow[from=1-1, to=2-1]
	\arrow[from=1-2, to=2-2]
	\arrow[hook, from=2-1, to=2-2].
\end{tikzcd}\]

Further, there is an action of $W$ on $H_*(\mathcal{F}\ell)$ and we can extend the identification in \cite{LusRemarks} with an identification of a locally closed $Z=\bigcup_{w, z} T_{Y_w^z}^*(G/B)\hookrightarrow \mathcal{F}\ell_{ts}$. We can then check the following statement.
\begin{claim}
    The image of the functor $j_!:\Dg\rightarrow D_{(P_-,\psi)}(\mathcal{F}\ell)$ lies in $\widetilde{\mathcal{D}}_\psi$.
\end{claim}
This can be thought of as a categorical interpretation of the geometric embedding in \cite{LusRemarks} together with the group action of $W$ on the categories involved. We can then state a conjecture relating Theorem \ref{main:geom} to the results in \cite{KLCatO}, which motivated our consideration of the category $\Dg$.

\begin{conj}
    The functor $j_!:\Dg\rightarrow \widetilde{\mathcal{D}}_\psi$ is Koszul dual to the functor $\mathrm{KL}_\mathcal{O}\rightarrow \overset{\bullet}{\mathfrak{u}} _q\mathrm{-mod}_{\hat{0}}$ constructed in \textup{\cite[\S 7]{KLCatO}}.
\end{conj}

\subsection{Overview of the paper}

We begin in Section \ref{sec:algebraic} by introducing $\Afmod$ and its derived version $\Da$. We will show that this category is natural to consider as it can be defined simply in two different ways: as the extension-closure of all of the $W$-many versions of $\O$, or by the axioms in Proposition \ref{prop:introafdef}. Key to our subsequent arguments will be the introduction in this section of ``restriction functors'' from $\Afmod$ to $\O$ which are adjoint to the natural inclusions of $W$-conjugates of $\O$ into $\Afmod$.

We then explain geometric analogs of these categories and functors in Section \ref{sec:geometric}. We explain geometric versions of $\Da$ and $\Afmod$ in terms of constructible and perverse sheaves on $G/B$, and we explain that the restriction functors established in the preceding section have natural interpretations in terms of pushforward and pullback along the projection $G/B \to U\backslash G/B$.

In Section \ref{sec:bb}, we recall the Beilinson--Bernstein localization equivalence and prove the equivalences between the algebraic and geometric objects considered thus far. We then explain how to build gluing functors (in the sense of the Kazhdan--Laumon construction), which are endofunctors of the derived version of $\O$ built by composing the inclusion and restriction functors introduced in the previous two sections. We explain that the gluing functors can be interpreted algebraically as Zuckerman functors or geometrically in terms of convolution with standard and costandard objects. The central computation of this section establishes this geometric interpretation which is crucial to the comparison of our categories with the Kazhdan--Laumon Category $\O$, which is defined in \cite{KLCatO} in terms of these standard and costandard convolutions.

In Section \ref{sec:klgluing} we recall the definition of this category $\KL$. We explain Koszul duality in the context of $\KL$ and define the Koszul dual category $\KL^\vee$. In Section \ref{sec:main}, we then combine all of the above to show that $\Da$, $\Dg$, and $\KL^\vee$ are all equivalent as dg categories. We explain that $\KL^\vee$ is the category of algebras over a certain monad, and thus the main tool in this section is the Barr--Beck--Lurie theorem, which we apply to an adjoint pair of functors built from the inclusion and restriction functors introduced in Sections \ref{sec:algebraic} and \ref{sec:geometric}. We also examine the interaction of $t$-structures with the derived equivalence we establish.

In Section \ref{sec:explicit} we provide some examples of objects in $\Da$, $\Dg$, and $\KL^\vee$, explaining details and consequences of the main equivalence in some explicit cases. These examples lead us to the construction in Section \ref{sec:coherent}, where we explain how to use the perspectives in the present paper to geometrically realize the coherent families of weight modules studied by Fernando and Mathieu in \cite{Fernando}, \cite{Mathieu}. We hope that this will be a first step in recasting their important classification of weight modules with finite-dimensional weight spaces in the context of algebraic geometry.

\subsection{Acknowledgments} We would like to thank Ben Gammage, Matt Hogancamp, Michael McBreen, Raphael Rouquier, and Ben Webster for useful discussions and for sharing many insights which were very helpful at various stages of the project.

\section{The category of finite weight modules}\label{sec:algebraic}

\subsection{Preliminaries}

Let $G$ be a reductive algebraic group over $\mathbb{C}$ with Lie algebra $\mathfrak{g}$. We fix a maximal torus $T \subset G$ with corresponding subalegbra $\mathfrak{h} \subset \mathfrak{g}$, and a Borel subgroup $B \subset G$ containing $T$ with corresponding subalgebra $\mathfrak{b}$. Let $U$ be the unipotent radical of $B$ and $\mathfrak{n}$ its Lie subalgebra, the nilpotent radical of $\mathfrak{b}$. 

We let $W = N_G(T)/T$ be the Weyl group. We let $S \subset W$ denote the set of simple reflections in $W$. We choose once and for all a set of lifts $\dot{w} \in N_G(T)$ which satisfy the Weyl group relations. We let $\Lambda$ be the weight lattice and $Q \subset \Lambda$ the root sublattice generated by the positive roots $R^+ \subset R$, where $R$ is the corresponding root system.  Let $\Pi$ be the set of simple roots.

\subsection{The BGG Category \texorpdfstring{$
\mathcal{O}$}{O} and a generalization}

\subsubsection{The BGG Category \texorpdfstring{$\mathcal{O}$}{O}}
The main starting point for the objects studied in the present paper is the principal block $\BGGO$ of the Bernstein-Gel'fand-Gel'fand Category $\mathcal{O}$, defined in \cite{BGG} as follows.
\begin{definition}\label{def:cato}
    Let $\BGGO \subset \gmod$ be the full subcategory of modules $M$ satisfying the following conditions.
    \begin{enumerate}
        \item $M$ is a finitely-generated $\mathcal{U}(\mathfrak{g})$-module.
        \item The center $Z(\mathfrak{g}) \subset \mathcal{U}(\mathfrak{g})$ acts locally nilpotently on $M$.
        \item $M$ is $\mathfrak{h}$-semisimple, that is, $M$ is a weight module \[M = \bigoplus_{\lambda \in Q} M_\lambda.\]
        \item $M$ is locally $\mathfrak{n}$-finite.
    \end{enumerate}
\end{definition}
We note that if one assumes $\mathfrak{h}$-semisimplicity (with arbitrary weights) along with the other three axioms, it follows from the basic structure theory of $\BGGO$ that all nonzero weight spaces lie in $Q \subset \Lambda$, and so we choose to emphasize this in the definition. 

A common technical variant of $\BGGO$ which is more directly amenable to the geometric context we will soon introduce relaxes the semisimplicity condition on the action of $\mathfrak{h}$, instead requiring that any element of $\mathfrak{h}$ acts on a weight space by a Jordan block with eigenvalue determined by the weight, while requiring that $Z(\mathfrak{g})$ acts trivially rather than locally nilpotently.

\begin{definition}\label{def:catovariant}
    Let $\O \subset \gmod$ be the full subcategory of modules $M$ satisfying the following conditions.
    \begin{enumerate}
        \item $M$ is a finitely-generated $\mathcal{U}(\mathfrak{g})$-module.
        \item The center $Z(\mathfrak{g}) \subset \mathcal{U}(\mathfrak{g})$ acts trivially on $M$.
        \item $M$ has a decomposition $$M = \bigoplus_{\lambda \in Q} M_\lambda$$
        such that $\mathfrak{h}$ acts locally finitely on each $M_\lambda$ with generalized eigenvalue $\lambda$.
        \item $M$ is locally $\mathfrak{n}$-finite.
    \end{enumerate}
    We also write $\tO$ for the subcategory of modules satisfying all of the above conditions except the first (finite generation) condition. 
\end{definition}

\subsubsection{Finite weight modules}
We now introduce a larger subcategory of $\gmod$ which consists of all weight modules with a trivial action of $Z(\mathfrak{g})$ along with finite-dimensional weight spaces.
\begin{definition}
    Let $\Afmod \subset \gmod$ be the full subcategory of modules $M$ satisfying the following conditions.
    \begin{enumerate}
        \item $M$ is a finitely-generated $\mathcal{U}(\mathfrak{g})$-module.
        \item The center $Z(\mathfrak{g}) \subset \mathcal{U}(\mathfrak{g})$ acts trivially on $M$.
        \item  $M$ has a decomposition $$M = \bigoplus_{\lambda \in Q} M_\lambda$$
        such that $\mathfrak{h}$ acts locally finitely on each $M_\lambda$ with generalized eigenvalue $\lambda$.
        \item Each weight space $M_\lambda$ is finite-dimensional.
    \end{enumerate}
    We let $\gmod_0$ denote the larger subcategory of $\mathcal{U}(\mathfrak{g})$-modules on which the center $\mathcal{Z}(\mathfrak{g})$ acts trivially.
    We call $\Afmod$ the category of \emph{finite weight modules} in $\gmod_0$.
\end{definition}

\subsubsection{The categories $\O^w$}

For any $w \in W$, we can define a Borel subalgebra $\mathfrak{b}^w = w\mathfrak{b}w^{-1}$ containing $\mathfrak{h}$, and $\{\mathfrak{b}^w\}_{w \in W}$ is the complete set of Borel subalgebras of $\mathfrak{g}$ containing $\mathfrak{h}$. We write $\mathfrak{n}^w = w\mathfrak{n}w^{-1}$ which is the nilpotent radical of $\mathfrak{b}^w$.

\begin{definition}
    For any $w \in W$, let $\O^w \subset \gmod$ be the full subcategory of modules $M$ satisfying the conditions of Definition \ref{def:catovariant} but with the role of $\mathfrak{n}$ replaced by $\mathfrak{n}^w$.
\end{definition}

By the PBW theorem, $\mathfrak{n}^w$-local finiteness of a finitely generated module $M$ implies that all weight spaces are finite-dimensional, a fact which we record in the following proposition.
\begin{prop}\label{prop:iotaw}
    For any $w \in W$, there is a fully faithful exact inclusion functor $\iota_w: \O^w \to \Afmod$. 
\end{prop}

\subsubsection{Simple objects in $\Afmod$}

We now prove the following classification result, which follows from combining the classification of weight modules in \cite{Fernando} and \cite{Mathieu} with the trivial central character hypothesis.
\begin{prop}\label{prop:afsimples}
    For any simple object $L \in \Afmod$, there exists some $w \in W$ such that $L \in \O^w$.
\end{prop}
\begin{proof}
    By the definition of $\Afmod$, $L$ is a simple weight module with finite-dimensional weight spaces. Such modules are classified in \cite{Fernando} and \cite{Mathieu} as follows. In \cite{Fernando}, it is shown that for any such simple module, there exists a unique parabolic subalgebra $\mathfrak{p} \subset \mathfrak{g}$ with unipotent radical $\mathfrak{u}$ and Levi subalgebra $\mathfrak{l}$ such that $L^\mathfrak{u}$ is irreducible and $L$ is the unique irreducible quotient of $\mathrm{Ind}_{\mathfrak{p}}^{\mathfrak{g}}(L^{\mathfrak{u}})$. If $\mathfrak{l} = \mathfrak{g}$ is the only Levi for which this can occur, $L$ is called \emph{cuspidal}. In the above setup, we can always choose a parabolic such that $L^{\mathfrak{u}}$ is cuspidal. We will show that with the additional assumptions that $L$ has trivial central character and all nonzero weight spaces $L_\lambda$ have $\lambda \in Q$, we must have $\mathfrak{l} = \mathfrak{h}$, so that $L$ is the quotient of a module induced from a Borel subalgebra as above. Such an induced module is by definition a Verma module in $\O^w$ for some $w \in W$, whose unique irreducible quotient also lies in $\O^w$, which will then complete the proof of the proposition.
    
    We note that for any $\mathfrak{p}$ as above, $L^{\mathfrak{u}}$ as a module over $\mathcal{U}(\mathfrak{l})$ must have central character $\chi_{\mathfrak{l}} \in \mathfrak{h}^*/(W_{\mathfrak{l}}, \cdot)$ corresponding to the orbit of some element whose containing $(W, \cdot)$-orbit corresponds to the trivial central character; in particular, it corresponds to the orbit of some element of the root lattice $Q_{\mathfrak{l}}$ for the Levi subalgebra $\mathfrak{l}$. It remains then to show that there can be no cuspidal module over any Levi subalgebra $\mathfrak{l} \supsetneq \mathfrak{h}$ with central character corresponding to an element of $Q_{\mathfrak{l}}$ which is supported on weights lying in $Q_{\mathfrak{l}}$.
    
    Suppose for contradiction that such a cuspidal module $L'$ over $\mathcal{U}(\mathfrak{l})$ existed. By \cite{Mathieu}, there exists some integer $d \geq 0$ such that for all $\lambda$, $\dim L_\lambda \leq d$. By Proposition 4.8 of loc.\ cit., there exists a unique semisimple coherent extension (c.f. Section 4 of loc.\ cit.) $M$ of $L'$ of degree $d$. In particular, this means $M[0] \cong L'$ (we use the notation of loc.\ cit., where $M[\lambda]$ denotes the submodule of $M$ consisting of the direct sum of weight spaces $M_{\mu}$ for $\mu \equiv \lambda \mod Q$). By Proposition 6.2(ii) of loc.\ cit., there is some infinite-dimensional admissible highest weight module (with respect to any Borel subgroup $\mathfrak{b}$ of our choice) $L(\lambda) \subset N[\lambda \mod Q_\mathfrak{l}]$. By Proposition 4.8(iii) of loc.\ cit., it must again have central character in $Q_\mathfrak{l}$. We know by the structure theory of $\O$ that any such highest weight module has $\lambda$ in the $(W_\mathfrak{l}, \cdot)$-orbit of $0$, which in particular lies in the root lattice $Q_\mathfrak{l}$. This means $\lambda \equiv 0\mod Q_\mathfrak{l}$, and so $L(\lambda)$ is actually a composition factor of $M[0]$. This is impossible, since $M[0] \cong L'$ is cuspidal.
\end{proof}

Thus Proposition \ref{prop:afsimples} is what gives the equivalence between the two definitions of $\mathcal{A}_f$ provided in the introduction.
\begin{corollary}\label{cor:afdefs}
    $\mathcal{A}_f$ is the full subcategory of $\gmod$ generated under extensions by objects in $\tO^w$, $w \in W$. 
\end{corollary}
One can then enumerate the number of simple objects in $\mathcal{A}_f$ as follows.
\begin{corollary}\label{cor:countsimples}
    The number of simple objects in $\Afmod$ is
    \begin{equation*}
\sum_{w \in W} |W/P(w)|
    \end{equation*}
where for any element $w\in W$, $P(w)$ is the standard parabolic subgroup of $W$ generated by all simple reflections $s$ with $\ell(sw) > \ell(w)$.
\end{corollary}
\begin{proof}
    By Proposition \ref{prop:afsimples}, the set of isomorphism classes of simple objects in $\Afmod$ is the orbit under the $W$-action (where $y \in W$ acts on a module by conjugating the $\mathfrak{g}$-action by $\dot{y}$) of the set of isomorphism classes of simple objects in $\O$. For $w \in W$, we write $L_w \in \O$ for the simple object with highest weight $w \cdot (-2\rho)$. One can then check that the stabilizer of $L_w$ under this $W$-action on isomorphism classes of simple objects in $\Afmod$ is exactly $P(w)$, and so the result follows from the orbit-stabilizer theorem. 
\end{proof}

\subsection{Singular support, algebraically}

\subsubsection{The associated variety of a $\mathfrak{g}$-module}

We now recall the definition of the \emph{associated variety} associated to any $\mathfrak{g}$-module $M$. We first note that $\mathcal{U}(\mathfrak{g})$, equipped with the PBW filtration, has associated graded $\mathrm{gr}~\mathcal{U}(\mathfrak{g}) \cong \mathrm{Sym}(\mathfrak{g}) \cong \mathbb{C}[\mathfrak{g}^*]$. 

\begin{definition}
    Let $M$ be a finitely generated $\mathcal{U}(\mathfrak{g})$-module. A \emph{good filtration} on $M$ is an ascending filtration of $\mathbb{C}$-vector spaces
    \[
        0 = M^0 \subset M^1 \subset \cdots \subset M
    \]
    such that:
    \begin{enumerate}
        \item $\mathcal{U}^i(\mathfrak{g}) \cdot M^j \subset M^{i+j}$ for all $i, j \geq 0$, where $\mathcal{U}^i(\mathfrak{g})$ is the part of the PBW filtration on $\mathcal{U}(\mathfrak{g})$ consisting of elements with degree at most $i$,
        \item $\bigcup_{i} M^i = M$,
        \item The associated graded module $\mathrm{gr}\,M = \bigoplus_{i} M^i / M^{i-1}$ is finitely generated as a module over $\mathrm{gr}~ \mathcal{U}(\mathfrak{g}) \cong \mathrm{Sym}(\mathfrak{g})$.
    \end{enumerate}
\end{definition}

\begin{definition}
    Let $M$ be a finitely generated $\mathcal{U}(\mathfrak{g})$-module equipped with a good filtration $(M^i)_{i \geq 0}$. The \emph{associated variety} $\mathrm{AV}(M) \subset \mathfrak{g}^*$ is defined as the support of the associated graded module $\mathrm{gr}\,M$ considered as a finitely generated $\mathrm{Sym}(\mathfrak{g}) \cong \mathbb{C}[\mathfrak{g}^*]$-module. That is,
    \[
        \mathrm{AV}(M) = \mathrm{Supp}_{\mathbb{C}[\mathfrak{g}^*]}(\mathrm{gr}\,M) \subset \mathfrak{g}^*.
    \]
    The associated variety is independent of the choice of good filtration.
\end{definition}

\begin{definition}
    Let $I_+ \subset \mathbb{C}[\mathfrak{g}^*] \cong \mathrm{gr}~\mathcal{U}(\mathfrak{g})$ be the ideal generated by all positive degree elements of the subalgebra $\mathcal{U}(\mathfrak{g})^T$ of weight-zero elements. Then let $Z_+ \subset \mathfrak{g}^*$ be the vanishing set, so $Z_+ = V(I_+)$.
\end{definition}

For any $\mathfrak{g}$-module $M$, we say that $M$ has singular support in $Z_+$ if $\mathrm{AV}(M) \subset Z_+$.

\begin{proposition}\label{prop:ugt}
    For a module $M \in \mathfrak{g}\mathrm{-mod}$, the action of $\mathcal{U}(\mathfrak{g})^T$ on $M$ is locally finite if and only if $M$ has singular support in $Z_+$. If $M$ is finite-length, then this occurs if and only if $M$ has finite-dimensional weight spaces.
\end{proposition}
\begin{proof}
    Suppose that the action of $\mathcal{U}(\mathfrak{g})^T$ on $M$ is locally finite. If $(M^i)_{i \geq 0}$ is a good filtration for $M$, this means that for any $m \in M$ and positive degree element $a \in \mathcal{U}(\mathfrak{g})^T$, there exists some $n$ for which $a^k \in M^n$ for all $k \geq 0$. Since $a$ has positive degree this means the action of the image of $a^n$ in $\mathbb{C}[\mathfrak{g}^*]$ on $\overline{m} \in \mathrm{gr}~M^\bullet$ is zero. Since $\mathrm{gr}~M^\bullet$ is finitely-generated, taking the maximum of the $n$ obtained in this way over a generating set implies that some power of $a$ lies in $\mathrm{Ann}(\mathrm{gr}~M^\bullet)$. This means the radical of $\mathrm{Ann}(\mathrm{gr}~M^\bullet)$ contains $I_+$, and so $\mathrm{gr}~M^\bullet$ is supported in $V(I_+)$; the converse follows similarly.
    
    Now note that if $M$ is a cyclic module, clearly local finiteness of the $\mathcal{U}(\mathfrak{g})^T$ action is equivalent to the finite-dimensionality of all weight spaces, since $\mathcal{U}(\mathfrak{g})^T$ manifestly preserves weight spaces. If $M$ is finite-length, then applying this to a Jordan-H{\"o}lder filtration for $M$ gives the final claim in the proposition.
\end{proof}

We now give an alternative natural description of the subvariety $Z_+$.
\begin{lemma}\label{lem:zunion}
    The variety $Z_+$ can be described as a union of nilpotent subalgebras; we have
    \begin{align}
        Z_+ = \bigcup_{w \in W} (\mathfrak{b}^\perp)^w \subset \mathfrak{g}^*.
    \end{align}
    After identifying $\mathfrak{g}^* \cong \mathfrak{g}$ under the Killing form, we have
    \begin{align}
         Z_+ = \bigcup_{w \in W} \mathfrak{n}^w \subset \mathfrak{g}.
    \end{align}
\end{lemma}
\begin{proof}
    We will show each containment for the second formulation. First, suppose that $x \in w\mathfrak{n}w^{-1}$ for some $w \in W$. Without loss of generality (and by choosing an alternative set of positive roots by conjugating the original set by $w$), we can simply assume that $x \in \mathfrak{n}$. We note that $I_+$ is spanned by monomials which are products of $h_\alpha$ for some $\alpha \in \Pi$, $e_\beta$ for some $\beta \in R^+$ and $f_\gamma$ for some $\gamma \in R^+$ such that the sum of the $\beta$ which appear minus the sum of the $\gamma$ which appear is zero. Since $x \in \mathfrak{n}$, terms with an $h_\alpha$ will all vanish. All other terms contain at least one nonzero $f_\gamma$ for $\gamma \in R^+$, and therefore vanish on $\mathfrak{n}$. This implies that $x \in Z_+$.

    On the other hand, suppose $x \in Z_+$. Suppose that there is no $w \in W$ such that $wxw^{-1} \in \mathfrak{n}$. This means the set of $\beta \in R$ such that $f_\beta(x) \neq 0$ is not contained in any subset of $R$ which serves as a subset of positive roots (otherwise, such a subset would be a $W$-conjugate of $R^+$, and we would have $wxw^{-1} \in \mathfrak{n}$ for some $w$). This means there exists some $\{\beta_i\}_i \subset R$ such that $f_{\beta_i}(x) \neq 0$ for every $\beta_i$ in the subset, such that there exists some nontrivial $\mathbb{Z}_{\geq 0}$-linear relation between the $\beta_i$. Writing
    \begin{align*}
        \sum_i c_i\beta_i + \sum_j d_j\beta_j = 0,
    \end{align*}
    by separating the positive roots $\beta_i$ from the negative roots $\beta_j$, we then consider the monomial
    $$\prod_i e_{\beta_i}^{c_i}\prod_j f_{\beta_j}^{d_j},$$
    which lies in $I_+$ since the linear relation above confirms that it has weight zero. On the other hand, its value on $x$ is not zero by construction, contradicting the fact that $x \in Z_+$.
\end{proof}

\subsection{Functors and derived categories}

In this section, we define a collection of functors relating the abelian categories $\Afmod$ and $\O$, as well as certain derived analogues.
\begin{definition}
    Let $\Da_0$ be the full dg subcategory of the bounded derived category $D^b(\gmod_0)$ generated by the essential image of the natural functor $D^b(\Afmod) \to D^b(\gmod_0)$.
\end{definition}

\begin{definition}\label{def:algwaction}
    For any $w \in W$, let $\theta_w : \gmod_0 \to \gmod_0$ denote the exact equivalence of categories obtained by conjugating the $\mathfrak{g}$-action by $\dot{w} \in N_G(T)$. We continue to use the notation $\theta_w$ to refer to the corresponding automorphism of $\Da_0$ when the context is clear.
\end{definition}

\subsubsection{Ascending and descending chains in $\tO$}
The following technical lemma is required to properly define the remaining functors.
\begin{lemma}\label{lem:zorn}
    Suppose $M \in \gmod_0$ admits an ascending chain of submodules $\{M_i\}_{i \geq 0}$ whose union is $M$, with each $M_i \in \tO$. Then $M \in \tO$. Further, if $M \in \gmod_0$ admits a descending chain of submodules $\{N_i\}_{i \geq 0}$ such that each quotient $M/N_i \in \tO$, then $M \in \tO$. 
\end{lemma}
\begin{proof}
    Suppose $M$ admits an ascending chain of submodules as in the lemma. Then for any $m \in M$, we can choose some $k$ such that $m \in M_k$. Since $M_k \in \tO$, let $\lambda$ be the weight of $m$ in $M_k$; then we know $\mathcal{U}(\mathfrak{n})\cdot m$ and $\mathcal{U}(\mathfrak{h})$ are finite-dimensional, where $\mathfrak{h}$ acts with eigenvalues determined by the character $\lambda$.
    
    Suppose $M$ admits a descending chain of submodules $\{N_i\}_{i \geq 0}$ such that each quotient $M/N_i$ lies in $\tO$. Let $m \in M$, then there is some $\lambda \in Q$ such that $p_i(m)$ is of weight $\lambda$ in each $M/N_i$ where $p : M \to M/N_i$ is the quotient map. Since elements of $\tO$ have nonzero weight spaces supported in $\mathbb{Z}_{\leq 0}R^+$, there exists some $d$ such that any element $a$ of $\mathcal{U}(\mathfrak{n})$ with PBW degree at least $d$ has $p_i(a \cdot m) = 0$ for all $i \geq 0$. This means $a \cdot m = 0$ for any such $a$, proving that $\mathfrak{n}$ acts locally nilpotently on $M$. Finally, one sees that the $\mathcal{U}(\mathfrak{g})$-submodule of $M$ generated by $m$ is a finitely-generated submodule on which $\mathfrak{n}$ acts locally nilpotently and $Z(\mathfrak{g})$ acts trivially.
    
    One then argues that such a module lies in $\O$ by induction on $\dim \mathcal{U}(\mathfrak{n})\cdot m$, which we just showed is finite-dimensional. Indeed, by local nilpotency of the $\mathfrak{n}$-action, one can find a vector $v_+ \in \mathcal{U}(\mathfrak{g})\cdot m$ annihilated by $\mathcal{U}(\mathfrak{n})$; since the center acts trivially, one can then see that $v_+$ is a weight vector, i.e.\ a vector on which $\mathfrak{h}$ acts by scalars. It follows that $\mathcal{U}(\mathfrak{g})\cdot v_+$ is the quotient of a Verma module with trivial central character, and therefore lies in $\O$. If $\mathcal{U}(\mathfrak{g})\cdot v_+ \subsetneq \mathcal{U}(\mathfrak{g}) \cdot m$, one continues by applying the same reasoning to the image of $m$ in the quotient module $(\mathcal{U}(\mathfrak{g}) \cdot m)/(\mathcal{U}(\mathfrak{g})\cdot v_+)$. Since $\dim \mathcal{U}(\mathfrak{n})\cdot m \subset M$ is finite-dimensional, this process must terminate.
\end{proof}
Of course, the result in Lemma \ref{lem:zorn} holds just as well for $\tO^w$ for any $w \in W$. 

\subsubsection{Inclusion and its adjoints}
Recall that in Proposition \ref{prop:iotaw}, we defined for any $w \in W$ a fully faithful and exact inclusion functor $\iota_w : \O^w \hookrightarrow \Afmod$, which we also can consider as a functor $\iota_w : \tO^w \to \gmod_0$. Recall that if $\theta_w$ is the functor defined in Definition \ref{def:algwaction}, it restricts to an isomorphism $\theta_w : \tO \cong \tO^w$.

\begin{definition}
    For any $w \in W$, we define
    \begin{align*}
        j_{w!} : \tO \to \gmod_0
    \end{align*}
    by $j_{w!} = \iota_w \circ \theta_w$. Since this is a fully faithful exact functor, we use the same notation for the derived functor $j_{w!} : D(\tO) \to D(\gmod_0)$.
\end{definition}

\begin{definition}
    For any $w \in W$, we let ${}^\circ j_w^! : \gmod_0 \to \tO$ and ${}^\circ j_w^* : \gmod_0 \to \tO$ be the left- and right-exact functors defined respectively as follows. For any $M \in \gmod_0$, let $M_w$ be the largest submodule of $M$ contained in $\tO^w$, and let $M^w$ be the largest quotient of $M$ contained in $\tO^w$. These modules are well-defined by Lemma \ref{lem:zorn}. We then let
    \begin{align*}
        {}^\circ j_w^!M & = \theta_{w^{-1}} M_w \in \tO,\\
        {}^\circ j_w^*M & = \theta_{w^{-1}} M^w \in \tO.
    \end{align*}
    Let $j_w^!, j_w^* : D(\gmod_0) \to D(\tO)$ be their right- and left-derived functors respectively. For the rest of the paper, we use the same notation to denote the restriction of these derived functors to $\Da_0$.
\end{definition}

\begin{proposition}\label{prop:originaladjunctions}
    For any $w \in W$, there are adjoint pairs of functors
    \begin{align*}
        (j_{w!}, {}^\circ j_w^!), & \qquad ({}^\circ j_w^*, j_{w!})
    \end{align*}
    between $\tO$ and $\gmod_0$, and
    \begin{align*}
        (j_{w!}, j_w^!), & \qquad (j_w^*, j_{w!}).
    \end{align*}
    between $D(\tO)$ and $\Da_0$.
\end{proposition}
\begin{proof}
    Suppose that $M \in \tO$ and $N \in \gmod_0$ where $f : j_{w!}M \to N$ is a morphism in $\gmod_0$. Since $\mathfrak{n}$ and $\mathfrak{h}$ act locally finitely on $M$, they must also act locally finitely on the image of $M$ under $f$. Similarly, the weight space decomposition of $M$ implies a weight space decomposition on the image of $f$. This implies that the image of $f$ is a submodule of $N$ lying in $\tO$, which by definition is a submodule of ${}^\circ j_w^!N$, thereby inducing a map $M \to {}^\circ j_w^!N$. Composing with the natural inclusion ${}^\circ j_w^!N \to N$ gives the other direction. A similar argument proves the proposition for the adjunction $({}^\circ j_w^*, j_{w!})$.
    
    By the adjunction theorem in \cite{Quil}, this implies the stated adjunctions of derived functors.
\end{proof}

\subsubsection{Example of an object in $\Afmod$ not lying in any $\O^w$}

\begin{example}
    In the case $\mathfrak{g} = \mathfrak{sl}_2$ where $W = \{e, s\}$, we now give a basic example of a module lying in $\mathcal{A}_0$ which does not lie in $\O$ or $\O^s$.  In this figure, each dot with an even integer label denotes a weight space of dimension one with a given weight, while arrows denote isomorphisms of vector spaces given by the action of the usual elements $e$ and $f$ in $\mathfrak{sl}_2$
    
    
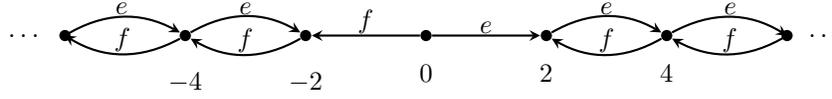
\begin{center}
        \begin{tikzpicture}[scale=0.8, every node/.style={circle, fill=black, inner sep=1.5pt}, >=stealth]
            
            \foreach \x in {-6, -4,-2,0,2,4, 6}
            \node (w\x) at (\x,0) {};
            
            \foreach \x in {-4,-2,0,2,4}
            \node[fill=none, below=8pt] at (\x,0) {$\x$};
            
            \node[fill=none] at (-6.65,0) {$\dots$};
            \node[fill=none] at (6.65,0) {$\dots$};
            
            \draw[->, thick, black, bend right=0] (w0) to node[above, sloped, midway, fill=none, inner sep=0pt, outer sep=0pt] {\textnormal{$f$}} (w-2);
            \draw[->, thick, black, bend left=30] (w-2) to node[above, sloped, midway, fill=none, inner sep=0pt, outer sep=0pt] {\textnormal{$f$}} (w-4);
            \draw[->, thick, black, bend right=-30] (w-4) to node[above, sloped, midway, fill=none, inner sep=0pt, outer sep=0pt] {\textnormal{$e$}} (w-2);
            \draw[->, thick, black, bend right=-30] (w-4) to node[above, sloped, midway, fill=none, inner sep=0pt, outer sep=0pt] {\textnormal{$f$}} (-6,0);
            \draw[->, thick, black, bend right=-30] (-6,0) to node[above, sloped, midway, fill=none, inner sep=0pt, outer sep=0pt] {\textnormal{$e$}} (w-4);

            \draw[->, thick, black, bend left=0] (w0) to node[above, sloped, midway, fill=none, inner sep=0pt, outer sep=0pt] {\textnormal{$e$}} (w2);
            \draw[->, thick, black, bend right=-30] (w2) to node[above, sloped, midway, fill=none, inner sep=0pt, outer sep=0pt] {\textnormal{$e$}} (w4);
            \draw[->, thick, black, bend left=30] (w4) to node[above, sloped, midway, fill=none, inner sep=0pt, outer sep=0pt] {\textnormal{$f$}} (w2);
            \draw[->, thick, black, bend left=30] (w4) to node[above, sloped, midway, fill=none, inner sep=0pt, outer sep=0pt] {\textnormal{$e$}} (6,0);
            \draw[->, thick, black, bend left=30] (6,0) to node[above, sloped, midway, fill=none, inner sep=0pt, outer sep=0pt] {\textnormal{$f$}} (w4);
        \end{tikzpicture}
    \end{center}
    If $M$ denotes this module, then it is easy to check that 
    \begin{align*}
        {}^\circ j_e^* M & = \Delta(0) & {}^\circ j_e^! M & = \Delta(-2).
    \end{align*}
\end{example}

\section{Sheaves on the flag variety}\label{sec:geometric}

\subsection{Preliminaries}

Considering $G/B$ as a complex variety, we let $D(G/B)$ denote the weakly constructible derived category of sheaves on $G/B$, considered as a dg category. We can consider the constructible derived category $D^c(G/B)$, where stalks and costalks are required to be finite-dimensional, as a full dg subcategory $D^c(G/B) \subset D(G/B)$. We write $D^b(G/B)$ for the subcategory of $D^c(G/B)$ consisting of objects with finitely many nonzero cohomology objects.

We note that $D^b(G/B)$ is equipped with the perverse $t$-structure whose heart we denote by $\mathrm{Perv}(G/B)$. We also consider the $U$-equivariant weakly constructible derived category $D_U(G/B)$ which consists of elements of $D(G/B)$ stratified with respect to the Bruhat stratification $\{BwB\}_{w \in W}$ on $G/B$. Similarly, we write $\mathrm{Perv}_U(G/B)$ for the heart of $D_U^c(G/B)$, which is the full abelian subcategory of $\mathrm{Perv}(G/B)$ consisting of objects constructible with respect to the Bruhat stratification. We also use the identification $D_U(G/B) \cong D(U \backslash G/B)$ to the weakly constructible derived category on the stack $U \backslash G/B$.

We let $\mathrm{Perv}_T(G/B)$ be the category of $T$-equivariant perverse sheaves on $G/B$ with respect to the left multiplication action of $T$. We say that an object in $\mathrm{Perv}(G/B)$ is unipotently $T$-monodromic if it is a finite extension of objects lying in the subcategory $\mathrm{Perv}_T(G/B)$.

\subsection{\texorpdfstring{$U$}{U}-equivariant sheaves on the flag variety and generalizations}

We recall in Section \ref{sec:bb} that $\mathrm{Perv}_U(G/B)$ is a geometric model for $\O$. We now explain two technical variants of $\mathrm{Perv}_U(G/B)$ which give rise to monoidal actions on its right and left respectively. 

\subsubsection{Equivariant and monodromic categories}
We also consider the $B$-equivariant derived category $D_B(G/B)$, which unlike in the case of $D_U(G/B)$, is not the derived category of $\mathrm{Perv}_B(G/B)$. It admits a monoidal structure
\begin{align*}
    - * - : D_B(G/B) \times D_B(G/B)  \to D_B(G/B) 
\end{align*}
by convolution.

In \cite{BY}, it is explained that there exists a different monoidal category Koszul dual to $D_B(G/B)$ which acts on $D_U(G/B)$ on the left; this is the category of \emph{free-monodromic} sheaves $\hat{D}(\wqw{B}{G}{B})$. Although defined in loc.\ cit.\ for $\ell$-adic sheaves on a variety over a field of positive characteristic, such a category can be defined (and was used and studied in \cite{DhTa}) with base field and coefficient field $\mathbb{C}$; this is the version of the category we use in the present paper. There is a convolution product
\begin{align*}
    - * - : \hat{D}(\wqw{B}{G}{B}) \times \hat{D}(\wqw{B}{G}{B})  \to \hat{D}(\wqw{B}{G}{B})
\end{align*}
which again equips $\hat{D}(\wqw{B}{G}{B})$ with the structure of a monoidal category.

For any $w \in W$, let $j_w : \mathbb{C}^{\ell(w)} \cong BwB \subset G/B$ be the inclusion. In $\mathrm{Perv}_B(G/B)$, there are standard and costandard objects 
\begin{align*}
    \Delta_w & = j_{w!}\mathbb{C}[\ell(w)]\\
    \nabla_w & = j_{w*}\mathbb{C}[\ell(w)]
\end{align*}
Following \cite{DhTa}, letting $\mathcal{L}_w$ be the universal unipotently monodromic local system on $\mathbb{C}^{\ell(w)} \times T \cong BwU \subset G/U$; we let $\tilde{j}_w$ denote this inclusion. We can then define analogues of the above objects which now lie in $\hat{D}(\wqw{B}{G}{B})$ by
\begin{align*}
    \hat{\Delta}_w & = \tilde{j}_{w!}\mathcal{L}_w[\ell(w)]\\
    \hat{\nabla}_w & = \tilde{j}_{w*}\mathcal{L}_w[\ell(w)].
\end{align*}
We note that $\Delta_e$ (resp.\ $\hat{\Delta}_e$) is the monoidal unit for convolution in $D_B(G/B)$ (resp.\ $\hat{D}(\wqw{B}{G}{B})$). 

\subsubsection{Singular support and a new category}

In Section \ref{sec:bb} we recall the Riemann--Hilbert correspondence by which one can identify $\mathrm{Perv}(G/B)$ with regular holonomic $\mathcal{D}$-modules on $G/B$. By this equivalence there is therefore a well-defined notion of singular support for any object of $\mathrm{Perv}(G/B)$, which is a subset of $T^*(G/B)$. 

\begin{definition}
    Given $w \in W$, recall the Schubert cell $UwB = BwB \subset G/B$ and its closure, the Schubert variety $Y_w = \overline{BwB} \subset G/B$. For any $z \in W$, we write $Y_w^z$ for the $z$-translate of $Y_w$ under multiplication by $\dot{z}$, so that $Y_w^z = \overline{zBwB}$. We now let
    \begin{align*}
        Z & = \bigcup_{w, z} T_{Y_w^z}^*(G/B)
    \end{align*}
\end{definition}

\begin{definition}
    Let $\mathcal{B}_f \subset \mathrm{Perv}(G/B)$ be the the subcategory of unipotently $T$-monodromic perverse sheaves $\mathcal{F}$ with singular support in $Z$.
\end{definition}

\begin{proposition}
    There is a natural exact and faithful inclusion functor $$\mathrm{Perv}_U(G/B) \to \mathcal{B}_f.$$ 
\end{proposition}
\begin{proof}
    Any $\mathcal{F} \in \mathrm{Perv}_U(G/B)$ has singular support in 
    \begin{align*}
        \bigcup_{w \in W} T_{Y_w}^*(G/B) \subset Z,
    \end{align*}
    and is unipotently $T$-monodromic; it can therefore be considered as an object in $\mathcal{B}_f$.
\end{proof}

\begin{definition}
    For any $w \in W$, let $\mathcal{P}^w = \mathrm{Perv}_{U^w}(G/B)$, where $U = \dot{w}U\dot{w}^{-1}$. It is equivalent to $\mathrm{Perv}_{U^w}(G/B)$ under the functor $\sigma_{w*} \cong \sigma_{w^{-1}}^*$, where $\sigma_w$ is the left multiplication map $G/B \to G/B$ by $\dot{w}$. 
\end{definition}

\subsection{Derived setup and adjoint functors}

\begin{definition}
    Let $\Dg_0 \subset D^b(G/B)$ be the full dg subcategory generated by the image of the natural functor $D^b(\mathcal{B}_f) \to D^b(G/B)$.
\end{definition}

\begin{definition}
    For any $w \in W$, we define $\pi_w^* : D_U(G/B) \to D(G/B)$ to be the composition of the isomorphism $\sigma^*_w : D_{U}(G/B) \to D_{U^w}(G/B)$ with the natural inclusion into $D(G/B)$ (i.e.\ the forgetful functor for $U^w$-equivariance on the left). 
\end{definition}

\begin{definition}
    For any $w \in W$, we define $\pi^w_!, \pi_{w*} : D(G/B) \to D_U(G/B)$ to be the composition of the natural shriek and star pushforward functors for the projection morphism (which can be thought of as $U^w$-averaging) $\pi_! : D(G/B) \to D({U^w}\backslash G /B) = D_{U^w}(G/B)$ with the isomorphism $\sigma^*_w : D_{U^w}(G/B) \to D_U(G/B)$. 
\end{definition}

Since one can think of $\pi_e^*$ as the pullback $\pi^*$ under this same projection morphism $\pi : D(G/B) \to D(U \backslash G/B)$, we obtain the following result from the usual six-functor formalism.
\begin{proposition}
    For any $w \in W$, there are adjunctions
    \begin{align*}
        (\pi_!^w, \pi_w^*),\qquad (\pi_w^*, \pi_*^w)
    \end{align*}
    between $D_U(G/B)$ and $D(G/B)$.
\end{proposition}

\section{Equivalences via the localization theorem}\label{sec:bb}

\subsection{Beilinson--Bernstein localization}

In \cite{BBloc}, Beilinson and Bernstein introduced the localization functor, an equivalence of categories
\begin{align*}
    \Gamma : \mathcal{D}\mathrm{-mod}(G/B) \to \gmod_0.
\end{align*}
Further, the Riemann--Hilbert correspondence tells us that $\mathrm{Perv}(G/B)$ is equivalent to the full subcategory of regular holonomic $\mathcal{D}$-modules on $G/B$.

\subsubsection{Localization and derived versions for $\tO$}
Considering the category of $U$-equivariant $\mathcal{D}$-modules, which are then already guaranteed to be regular holonomic, one often considers a restriction of the localization functor 
\begin{align*}
    \Gamma : \mathrm{Perv}_U(G/B) \to \O.
\end{align*}
We will also use its derived version 
\begin{align*}
    \Gamma : D_U(G/B) \cong D(\tO),
\end{align*}
which follows from the observation in \cite{BGS} that $D_U(G/B)$ is naturally isomorphic to the derived category of $\mathrm{Perv}_U(G/B)$. 

\subsubsection{Localization for $\mathcal{B}_f$}

\begin{lemma}\label{lem:ssloc}
    Suppose that $\mathcal{F} \in \mathrm{Perv}(G/B)$ has singular support lying in $Z$. Then the associated variety of the corresponding module $\Gamma(\mathcal{F}) \in \gmod_0$ lies in $Z_+$. 
\end{lemma}
\begin{proof}
    Let $\mu : T^*(G/B) \to \mathfrak{g}^*$ be the moment map; it is a general fact about the Beilinson--Bernstein localization functor that if $\mathcal{F}$ has singular support in some $Z' \subset T^*(G/B)$, then the associated variety of $\Gamma(\mathcal{F})$ is contained in $\mu(Z')$.

    The result in the lemma will then follow from the fact that
    \[\mu^{-1}(Z_+)= Z.\]
    By Lemma \ref{lem:zunion}, we have
    \begin{align}
        \mu^{-1}(Z_+) = \bigcup_{w \in W} \mathfrak{n}^w.
    \end{align}
    By the definition of $Z$, it is then clear that this follows from the well-known fact that
    \begin{align}
        \mu^{-1}(\mathfrak{n}) = \bigcup_{w \in W} T_{Y_w}^*(G/B),
    \end{align}
    in other words, the preimage of $\mathfrak{n}$ under the moment map is exactly the union of the conormal varieties to all Schubert varieties.
\end{proof}

\begin{corollary}\label{cor:bfinaf}
    For any $\mathcal{F} \in \mathcal{B}_f$, the associated module $\Gamma(\mathcal{F})$ obtained under Beilinson--Bernstein localization lies in $\mathcal{A}_f$.
\end{corollary}
\begin{proof}
    Suppose $\mathcal{F} \in \mathcal{B}_f$. Since $\mathcal{F}$ is unipotently monodromic for the $T$-action of $G/B$ on the left, we know that $\Gamma(\mathcal{F})$ has a locally finite action of $\mathfrak{h}$ with eigenvalues determined by weights lying in the root lattice $Q$. This is clear by Beilinson--Bernstein localization for $T$-equivariant sheaves, and this condition is stable under finite extensions. Further, by Lemma \ref{lem:ssloc} and Proposition \ref{prop:ugt}, $\Gamma(\mathcal{F})$ has finite-dimensional weight spaces. This ensures that $\Gamma(\mathcal{F}) \in \mathcal{A}_f$ as desired.
\end{proof}

Corollary \ref{cor:bfinaf} implies the following equivalence, which we will continue to denote by $\Gamma$.
\begin{corollary}\label{cor:dggamma}
    The Beilinson--Bernstein localization functor induces an equivalence of dg categories
    \begin{align*}
        \Gamma : \Dg_0 \cong \Da_0. 
    \end{align*}
\end{corollary}

\subsection{Algebraic and geometric interpretations of functors}

We now explain how the functors we have introduced are intertwined by the localization equivalence $\Gamma$. To start, one obtains the following equivalences directly from the definitions.
\begin{lemma}\label{lem:functordict1}
    For any $w \in W$, the diagrams
    \[\begin{tikzcd}
        \Dg_0 \arrow[d, "\sigma_{w*}"] \arrow[r, "\Gamma"] & \Da_0 \arrow[d, "\theta_w"]\\
        \Dg_0 \arrow[r, "\Gamma"] & \Da_0
    \end{tikzcd}\qquad \begin{tikzcd}
        D_U(G/B) \arrow[d, "\pi_w^*"] \arrow[r, "\Gamma"] & D(\tO) \arrow[d, "j_{w!}"]\\
        \Dg_0 \arrow[r, "\Gamma"] & \Da_0
    \end{tikzcd}\]
    commute.
\end{lemma}

By uniqueness of left and right adjoints, one then obtains the following corollary.
\begin{corollary}\label{cor:functordict2}
    For any $w \in W$, the diagrams
    \[\begin{tikzcd}
        \Dg_0 \arrow[d, "\pi_*^w"] \arrow[r, "\Gamma"] & \Da_0 \arrow[d, "j_w^!"]\\
        D_U(G/B) \arrow[r, "\Gamma"] & D(\tO)
    \end{tikzcd}\qquad 
    \begin{tikzcd}
        \Dg_0 \arrow[d, "\pi_!^w"] \arrow[r, "\Gamma"] & \Da_0 \arrow[d, "j_w^*"]\\
        D_U(G/B) \arrow[r, "\Gamma"] & D(\tO)
    \end{tikzcd}\]
    commute.
\end{corollary}

\subsection{Gluing functors}

\subsubsection{Definition of gluing functors}

\begin{definition}\label{def:geogluing}
    For any $w \in W$, let
    \begin{align*}
        \Phi_w : D_U(G/B) & \to D_U(G/B) & \Psi_w : D_U(G/B) & \to D_U(G/B)\\
        \Phi_w & = \pi_*^w\pi_e^*& \Psi_w & = \pi^w_!\pi_e^*.
    \end{align*}
\end{definition}
\begin{definition}
    For any $w \in W$, let
    \begin{align*}
        F_w : D(\tO) & \to D(\tO) & G_w : D(\tO) & \to D(\tO)\\
        F_w & = j_w^!j_{e!}& G_w & = j_w^*j_{e!}.
    \end{align*}
    We write ${}^\circ F_w$ and ${}^\circ G_w$ for the abelian versions considered as endofunctors of $\O$.
\end{definition}
The following is then implied by Lemma \ref{lem:functordict1} and Corollary \ref{cor:functordict2}.
\begin{corollary}\label{cor:glueloc}
    The diagrams
\[\begin{tikzcd}
D_U(G/B) \arrow[r, "\Gamma"] \arrow[d, "\Phi_w"] & D(\tO)\arrow[d, "F_w"]\\
D_U(G/B) \arrow[r, "\Gamma"] & D(\tO)
\end{tikzcd} \qquad \begin{tikzcd}
D_U(G/B) \arrow[r, "\Gamma"] \arrow[d, "\Psi_w"] & D(\tO)\arrow[d, "G_w"]\\
D_U(G/B) \arrow[r, "\Gamma"] & D(\tO)
\end{tikzcd}\]
commute.
\end{corollary}

\subsubsection{Algebraic interpretation} We now explain an algebraic interpretation of the gluing functors $F_w$, $G_w$ in terms of \emph{Zuckerman functors}. These functors were defined in \cite{Zuckerman}, and an explanation more amenable to the setup of the present paper can also be found in \cite[Section 1]{MS}.

\begin{definition}
Given a simple reflection $s \in S$, let $\mathfrak{g}^s$ denote the corresponding minimal parabolic subalgebra of $\mathfrak{g}$ which contains $\mathfrak{h}\oplus \mathfrak{n}$. We denote by ${}^s\O \subset \O$ the corresponding parabolic subcategory, which consists of all locally $\mathfrak{g}^s$-finite objects in $\O$.
\end{definition}
\begin{definition}
    For any simple reflection $s \in S$, let 
    \begin{align*}
        Z_s & : \O \to \O,\\
        Z_s' & : \O \to \O
    \end{align*}
    be the functors which take a module $M \in \O$ to the maximal quotient of $M$ contained in ${}^s\O$, and the maximal submodule of $M$ contained in ${}^s\O$ respectively.
\end{definition}

One can then check the following from the definitions.
\begin{prop}\label{prop:oursiszuck}
    For any $s \in S$,
    \begin{align*}
        Z_s & = {}^\circ j_s^*j_{e!} = {}^\circ G_s,\\
        Z_s' & = {}^\circ j_s^!j_{e!} = {}^\circ F_s.
    \end{align*}
\end{prop}

Both variants of Zuckerman functors satisfy braid relations in $W$, see \cite{Zuckerman} or \cite{MS}, and thus by Proposition \ref{prop:oursiszuck}, these relations also hold for our functors ${}^\circ F_w$ and ${}^\circ G_w$.

\begin{remark}
    The main purpose of this section is to show that our dg functors $F_w$ and $G_w$ are ``derived versions" of a pair of easy-to-describe and well-studied functors which show up ubiquitously in the literature on Category $\O$. However, one must be careful about the precise technicalities used to construct these ``derived versions." This is because if one naively considers the derived functors of $Z_s$ and $Z_s'$, they differ significantly from our $F_s$ and $G_s$; the strange behavior of these derived functors in nonzero cohomological degrees is described in \cite[Theorem 2 and Remark 1.2]{MS}. On the algebraic side of the localization equivalence, this is because $j_s^*$ and $j_s^!$ are defined as derived functors from the larger source category $\gmod \supset \O$, and on the geometric side this discrepancy can be explained by the fact that the $B$-equivariant derived category and the derived category of $B$-equivariant sheaves do not coincide.
\end{remark}

\subsubsection{Geometric interpretation}
We now explain a key observation which will be crucial to our main theorem; this gives a geometric interpretation of the gluing functors $\Phi_w$ and $\Psi_w$ in terms of a familiar convolution action by standard and costandard objects (these convolutions are sometimes called \emph{intertwining functors}), albeit with a cohomological shift.
\begin{proposition}\label{prop:convinterpretation}
    There are natural isomorphisms of functors
    \begin{align*}
        \Phi_w(-) & \cong (\hat{\nabla}_w[-\ell(w)]) * -\\
        \Psi_w(-) & \cong (\hat{\Delta}_w[\ell(w)]) * - 
    \end{align*}
\end{proposition}
\begin{proof}

    We follow the notation of \cite{BY} and write $D(\qw{G}{B})$ for the category of unipotently $T$-monodromic sheaves on $G/U$ for the right $T$-action, and write $\tilde{\pi}_! : D(\qw{G}{B}) \to D(U \backslash \qw{G}{B}) \cong D(\wqw{B}{G}{B})$ for the $U$-averaging functor. 
    
    Since both $\tilde{\pi}_!$ and $\hat{\Delta}_s[1] * -$, considered as endofunctors of $D(\wqw{B}{G}{B})$, commute with the monoidal action of the category $\hat{D}(\wqw{B}{G}{B})$ by right-convolution, it suffices to prove the proposition in the case where $\mathcal{F}$ is the monoidal identity $\hat{\delta} \in \hat{D}(\wqw{B}{G}{B})$. One can then convolve on the right with an arbitrary object of $D_U(G/B)$ to obtain the stated natural isomorphisms of endofunctors of $D_U(G/B)$.
    
    In other words, since $\tilde{\pi}_!(\sigma_s^*(\tilde{\pi}^!\hat{\delta})) \cong \tilde{\pi}_!(\sigma_s^*(\hat{\delta}))[2]$, we need to show that
    \begin{align}
        \tilde{\pi}_!(\sigma_s^*(\hat{\delta}))[2] \cong \hat{\Delta}_s[1]\label{eqn:proconv}.
    \end{align}

    Let $j_e : T \hookrightarrow G/U$ and let $j_s : UsB \hookrightarrow G/U$ be the natural inclusions. Recall that $\mathcal{L}$ is the universal unipotently monodromic local system on $T \cong B/U \subset G/U$ while $\mathcal{L}_{s}$ is the universal unipotently monodromic local system on $UsB/U \subset G/U$ obtained by considering $UsB/U$ as a $T$-torsor over $UsB/B \subset G/B$. To show (\ref{eqn:proconv}), we must then show
    \begin{align}
        \tilde{\pi}_!(\sigma_s^*(j_{e!}\mathcal{L}))[2] \cong j_{s!}\mathcal{L}_{s}[2]\label{eqn:nconv}.
    \end{align}
    Letting $T^{\sigma_s}$ be the image of $T \subset G/U$  under $\sigma_s$ and $j_e^{\sigma_s}$ its inclusion into $G/U$, we then have a Cartesian diagram
    \[\begin{tikzcd}
        T^{\sigma_s} \arrow[r, "\cong"] \arrow[d, "j_e^{\sigma_s}"] & T \arrow[d, "j_e"]\\
        G/U \arrow[r, "\sigma_s"] &  G/U.
    \end{tikzcd}\]
    By the proper base change theorem, one then has
    \begin{align*}
        \tilde{\pi}_!(\sigma_s^*(j_{e!}\mathcal{L})) & = \tilde{\pi}_!\sigma_s^*j_{e!}\mathcal{L}\\
        & = \tilde{\pi}_!j_{e!}^{\sigma_s}\mathcal{L},
    \end{align*}
    where by abuse of notation $\mathcal{L}$ is now considered as a local system on $T^{\sigma_s}$. We then note that $(\tilde{\pi}\circ j_e^{\sigma_s})_!\mathcal{L} \cong j_{s!}\mathcal{L}_{s}$.

    By Verdier duality one then gets the dual result for $\pi_*$, yielding the second equation in the proposition.
\end{proof}

By Proposition \ref{prop:convinterpretation} and Corollary \ref{cor:glueloc}, we obtain the following. 
\begin{corollary}\label{cor:braids}
    If $s_{i_1} \dots s_{i_k}$ is a reduced word for $w \in W$, there are are natural isomorphisms of functors
    \begin{align*}
        \Phi_{s_{i_1}} \circ \dots \circ \Phi_{s_{i_k}} & \cong \Phi_{w},\\
        \Psi_{s_{i_1}} \circ \dots \circ \Psi_{s_{i_k}} & \cong \Psi_{w},\\
        F_{s_{i_1}} \circ \dots \circ F_{s_{i_k}} & \cong F_{w},\\
        G_{s_{i_1}} \circ \dots \circ G_{s_{i_k}} & \cong G_{w}.
    \end{align*}
\end{corollary}

\begin{proposition}\label{prop:nus}
    For any $s \in S$, there are natural transformations
    \begin{align*}
        \nu_s & : \Phi_s \circ \Phi_s \to \mathrm{Id},\\
        \nu_s' & : \mathrm{Id} \to \Psi_s \circ \Psi_s
    \end{align*}
    such that the natural transformations $\Phi_s \circ \Phi_s \circ \Phi_s \to \Phi_s$ given by $\nu_s \circ \Phi_s$ and $\Phi_s \circ \nu_s$ coincide. Similarly, the natural transformations $\Psi_s \to \Psi_s \circ \Psi_s \circ \Psi_s$ given by $\Psi_s \circ \nu_s'$ and $\nu_s' \circ \Psi_s$ concide.
\end{proposition}
\begin{proof}
    For any $s \in S$, there are distinguished triangles
    \[\begin{tikzcd}
        \hat{\delta}\arrow[r] & \hat{T}_s \arrow[r] &  \hat{\nabla}_s \arrow[r, "{[1]}"] & \quad \\
        \hat{\Delta}_s \arrow[r] & \hat{T}_s \arrow[r] & \hat{\delta} \arrow[r, "{[1]}"] & \quad
    \end{tikzcd}\]
    and therefore we have the shifted maps
    \begin{align*}
         \hat{\nabla}_s[-1] & \to \hat{\delta},\\
        \hat{\delta}_s & \to \hat{\Delta}_s[1],
    \end{align*}
    giving rise to natural transformations
    \begin{align*}
        \eta_s & :\Phi_s \to \mathrm{Id},\\
        \eta_s' & :\mathrm{Id} \to \Psi_s.
    \end{align*}
    We then let $\nu_s : \Phi_s \circ \Phi_s \to \mathrm{Id}$ be defined by $\nu_s = \eta_s\circ (\Phi_s\eta_s)$, $\nu_s' = (\Psi \eta_s') \circ \eta_s'$. The given associativity constraints are then implied automatically by the associativity constraints for the convolution product.
\end{proof}
The same result then also holds in the algebraic setting for the functors $F_s$ and $G_s$.

\section{The Kazhdan--Laumon gluing construction}\label{sec:klgluing}
In this section, we recall the Kazhdan--Laumon gluing construction from \cite{KL} and its application in \cite{KLCatO} to the setting of Category $\mathcal{O}$. We explain a natural way to define a category which is \emph{Koszul dual} to the one studied in loc.\ cit., setting up a comparison with the categories introduced and considered in the previous sections. 

\subsection{Overview of the Kazhdan--Laumon construction}

\subsubsection{Gluing of abelian categories}\label{subsec:gluingofabelian}
In the preceding sections, we constructed two sets of data (belonging to the algebraic and the geometric setup respectively) of the following form.
\begin{definition}\label{def:wgluing}
    A left \emph{$W$-gluing data} for abelian categories is:
    \begin{enumerate}
        \item A tuple of abelian categories $(\mathcal{C}_w)_{w \in W}$ (in all cases we will consider, each $\mathcal{C}_w$ will be identical),
        \item A collection of right-exact functors such that for every $w$,
        \begin{align*}
            \mathsf{F}_{w} : \mathcal{C}_{y} \to \mathcal{C}_{wy}
        \end{align*}
        for every $y \in W$.
        \item \label{item:naturals} For any $w, w' \in W$, a natural transformation
        \begin{align*}
            \nu_{w,w'} : \mathsf{F}_w \circ \mathsf{F}_{w'} \to \mathsf{F}_{ww'}
        \end{align*}
        which is an isomorphism when $\ell(w) + \ell(w') = \ell(ww')$, and which satisfy the associativity condition that for any triple $x, y, z \in W$, the equation
        \begin{align*}
            \nu_{w,yz}\circ (\mathsf{F}_w\nu_{y,z}) & = \nu_{wy,z}\circ (\nu_{w,y}\mathsf{F}_z).
        \end{align*}
    \end{enumerate}
\end{definition}
A \emph{right $W$-gluing data} can just as easily be defined with left-exact functors. In setup of the abelian version of the gluing functors considered in the present paper, the existence of natural transformations satisfying the properties of (\ref{item:naturals}) follows by Corollary \ref{cor:braids} and Proposition \ref{prop:nus}.

Given a left (or right) $W$-gluing data, one can define the \emph{glued category} with respect to this data as follows.
\begin{definition}\label{def:gluedcat}
    Given a left $W$-gluing data \[\{(\mathcal{C}_w)_{w \in W}, (\mathsf{F}_w)_{w \in W}, (\nu_{w,w'})_{w, w' \in W}\},\]
    the glued category associated to this data is the category of tuples of objects $(A_w)_{w \in W}$ with $A_w \in \mathcal{C}_w$ equipped with morphisms $\alpha_{y,w} : \mathsf{F}_yA_w \to A_{yw}$ in $\mathcal{C}_{yw}$ for every $y, w \in W$ such that the following diagram
    \[\begin{tikzcd}
        \mathsf{F}_{z}\mathsf{F}_yA_w \arrow[r, "\mathsf{F}_z\alpha_{y,w}"] \arrow[d, "\nu_{z,y}(A_w)"'] & \mathsf{F}_zA_{yw} \arrow[d, "\alpha_{z,yw}"]\\
        \mathsf{F}_{zy}A_w \arrow[r, "\alpha_{zy,w}"'] & A_{zyw}
    \end{tikzcd}\]
    commutes for any triple $z, y, w \in W$.
\end{definition}

The original and central example of Kazhdan--Laumon gluing which appeared in Kazhdan and Laumon's initial paper \cite{KL}, in which this gluing procedure is defined, proceeds as follows. They consider sheaves on the basic affine space $G/U$, but over the algebraic closure of a finite field $\mathbb{F}_q$. As a result, they work with $\ell$-adic sheaves for $\ell$ coprime to $q$ and have access to standard artifacts of this setting such as Artin--Schreier sheaves and Tate twists. In particular, we denote by $(\tfrac{1}{2})$ the ``half Tate twist" corresponding to some chosen square root of $q$.

In this context they construct for every $s \in S$ an endomorphism of the constructible derived category
\begin{align*}
    \mathrm{FT}_{s} : D^b(G/U, \overline{\mathbb{Q}}_\ell) \to D^b(G/U, \overline{\mathbb{Q}}_\ell);
\end{align*}
called a \emph{symplectic Fourier transform}. These functors have abelian versions which they consider as endomorphisms of $\mathrm{Perv}(G/U, \overline{\mathbb{Q}}_\ell)$. The bulk of \cite{KL} is devoted to showing that they satisfy braid relations and assemble to the data of a left $W$-gluing on $D^b(G/U, \overline{\mathbb{Q}}_\ell)$.

\subsubsection{Kazhdan--Laumon Category $\O$}

A simplification of this original example, which will be crucial to compare to the setting of the present paper, was explored in \cite{KLCatO}. Following the same philosophy that Category $\mathcal{O}$ is a fruitful place to start when studying the category of $\mathfrak{g}$-modules, one can consider the Kazhdan--Laumon gluing construction restricted to the subcategory
\[ \mathrm{Perv}(B\backslash G/B, \overline{\mathbb{Q}}_\ell) \subset \mathrm{Perv}(G/U, \overline{\mathbb{Q}}_\ell).\]

The result is the \emph{Kazhdan--Laumon category $\mathcal{O}$}, which we denote by $\mathrm{KL}_{\mathcal{O}}$. The first key observation in the study of this category is as follows.
\begin{lemma}\label{lem:ftisconv}
    There is a natural isomorphism
    \begin{align*}
        \mathrm{FT}_s|_{D^b(B\backslash G/B, \overline{\mathbb{Q}}_\ell)} \cong - * \nabla_s(\tfrac{1}{2}),
    \end{align*}
    where $(\tfrac{1}{2})$ denotes the half Tate twist on $\ell$-adic sheaves and $\nabla_s = j_{s*}\mathbb{C}[1](\tfrac{1}{2})$ for $j_s : BsB \hookrightarrow G/B$. 
\end{lemma}

One can then simplify the definition of the Kazhdan--Laumon category $\mathcal{O}$ without relying on the full definition of the symplectic Fourier transforms detailed in \cite{KL}, whose complexities are studied further in \cite{P}, \cite{217}.
\begin{definition}\label{def:klcato}
    The \emph{Kazhdan--Laumon category $\mathcal{O}$}, which we denote by $\KL$, is the $W$-glued category of tuples $(A_w)_{w \in W}$ for $A_w \in \mathrm{Perv}(B\backslash G/B)$ equipped with gluing functors \[(\nabla_{w}(\tfrac{\ell(w)}{2}) * - )_{w \in W}.\]
\end{definition}
Although the definition in \cite{KLCatO} and Lemma \ref{lem:ftisconv} are given in terms of the right convolution action by $\nabla_s(1)$, we prefer to define and work with the left action in the setting of the present paper (so as to more directly compare it with the other categories we consider), hence the chosen convention in Definition \ref{def:klcato}, which makes no difference due to the inherent symmetry of $\mathrm{Perv}(B\backslash G/B)$.

\subsection{Derived gluing and monads}

Although abelian categories appear in the original definitions of \cite{KL}, this same definitions have been made in the setting of derived and dg categories in \cite[Section 3.4]{P} and \cite[Section 4]{CMF3}.

Further, the perspective of \cite{BBP} gives a cleaner way to define $W$-glued categories. Given left $W$-gluing data as in Definition \ref{def:wgluing} (or even the analogous gluing data on a tuple of dg categories), it is explained in loc.\ cit.\ how the functors $(\mathsf{F}_w)_{w \in W}$ give rise to the data of a single functor
\begin{align*}
    \mathsf{F} : \bigoplus_{w \in W} \mathcal{C}_w \to \bigoplus_{w \in W} \mathcal{C}_w,
\end{align*}
while similarly the natural transformations $(\nu_{w,y})_{w, y \in W}$ give rise to a single natural transformation $\mathsf{F} \circ \mathsf{F} \to \mathsf{F}$. By the construction of $\mathsf{F}$ there is then additionally a transformation $\mathrm{Id} \to \mathsf{F}$ which equips $\mathsf{F}$ with the structure of a monad. We then get the following result.
\begin{prop}\label{prop:klismonad}
    Any left $W$-gluing data can be described as a monad on $\oplus_{w \in W} \mathcal{C}_w$ given by
    \begin{align*}
        \mathsf{F}((A_w)_{w \in W}) & = (\oplus_{y \in W} \mathsf{F}_yA_{y^{-1}w})_{w \in W}
    \end{align*}
    Similarly, any right $W$-gluing data can be described as a comonad on $\oplus_{w \in W} \mathcal{C}_w$ given by
    \begin{align*}
        \mathsf{G}((A_w)_{w \in W}) & = (\oplus_{y \in W} \mathsf{G}_yA_{y^{-1}w})_{w \in W}.
    \end{align*}
    These observation are true in both the abelian and derived setting.
\end{prop}

\subsection{Koszul duality}

\subsubsection{Koszul duality between $B \backslash G / B$ and $U \backslash G / U$\label{subsec:kdkl}}

Although in this paper we work with the free monodromic category $\hat{D}(\wqw{B}{G}{B})$ with complex coefficients and base field, it was originally defined in \cite{BY} in the context of $\ell$-adic sheaves with base field $\mathbb{F}_q$; in this subsection, we work $\ell$-adically and consider this category, denoting it by $\hat{D}(\wqw{B}{G}{B}, \overline{\mathbb{Q}_\ell})$ to avoid ambiguity. Similarly, we work also with the bi-equivariant derived category $D(B\backslash G/B, \overline{\mathbb{Q}_\ell})$.

One of the main contributions of \cite{BY} is the following result.
\begin{theorem}[\cite{BY}]
    There exists a \emph{Koszul duality} functor
    \begin{align*}
        \mathbb{K} : D^b(B\backslash G/B, \overline{\mathbb{Q}_\ell}) \to \hat{D}&^b(\wqw{B}{G}{B}, \overline{\mathbb{Q}_\ell})
    \end{align*}
    satisfying the following properties:
    \begin{enumerate}
        \item $\mathbb{K}$ is an equivalence.
        \item If $(\frac{1}{2})$ denotes the half Tate twist introduced in Section \ref{subsec:gluingofabelian}, then for any $\mathcal{F} \in D^b(B\backslash G/B, \overline{\mathbb{Q}_\ell}),$
        \begin{align*}
            \mathbb{K}(\mathcal{F}(\tfrac{1}{2})) & = \mathcal{F}(-\tfrac{1}{2})[-1].
        \end{align*}
        \item For any $w \in W,$
        \begin{align*}
            \mathbb{K}(\Delta_w) & = \hat{\Delta}_w,\\
            \mathbb{K}(\nabla_w) & = \hat{\nabla}_w.
        \end{align*}
    \end{enumerate}
\end{theorem}

One can then check that for any derived $W$-gluing data on $D^b(B\backslash G/B, \overline{\mathbb{Q}_\ell})$, applying the equivalence $\mathbb{K}$ gives rise to a derived $W$-gluing data on the category $\hat{D}^b(\wqw{B}{G}{B}, \overline{\mathbb{Q}_\ell})$. We then observe that for any $s \in S$,
\begin{align*}
    \mathbb{K}(\nabla_s(\tfrac{1}{2})) & = \hat{\nabla}_s(-\tfrac{1}{2})[-1].
\end{align*}
This means that the category $\hat{D}^b(\wqw{B}{G}{B}, \overline{\mathbb{Q}_\ell})$ can be equipped with a gluing data such that for each $s \in S$, the corresponding gluing functor is given by\[\mathcal{F} \mapsto \hat{\nabla}_s(-\tfrac{1}{2})[-1] * \mathcal{F}.\]

For simplicity, we note that we can just as easily apply this gluing construction, with the very same gluing functors, to the category $D_U^b(G/B, \overline{\overline{\mathbb{Q}}_{\ell}})$, since this category also admits an action on the left by the objects above under convolution. We call this the \emph{mixed Koszul dual Kazhdan--Laumon category $\mathcal{O}$}.

\subsubsection{The Koszul dual Kazhdan--Laumon category $\mathcal{O}$}

Returning now to the setup over $\mathbb{C}$, we can then use the mixed setting as motivation to naturally define a complex version of the preceding construction.
\begin{definition}
    The \emph{Koszul dual Kazhdan--Laumon category $\mathcal{O}$}, which we denote by $\mathrm{KL}_{\mathcal{O}}^{\vee}$, is the Kazhdan--Laumon category built from $\oplus_{w \in W} D_U(G/B)$ and the gluing functors $\hat{\nabla}_w[-\ell(w)]*-$.
\end{definition}

In other words, the gluing data $(\Phi_{w})_{w \in W}$ constructed in Definition \ref{def:geogluing} is exactly the gluing data comprising the functors in the definition of $\mathrm{KL}_{\mathcal{O}}^\vee$.

\section{The main equivalence}\label{sec:main}

In this section, we make some technical modifications to the dg categories $\Da$ and $\Dg$ so that we may apply the Barr--Beck--Lurie theorem, the main tool which we will use to show the equivalence between these categories and $\KL^\vee$.

\subsection{Functors and cocompletions}

\begin{definition}
    Define the functors
    \begin{align*}
        j^! & : D(\gmod_0) \to \bigoplus_{w \in W} D(\tO), & j^!(M) & = (j_w^! M)_{w \in W},\\
        j_! & : \bigoplus_{w \in W} D(\tO) \to  D(\gmod_0), & j_!((A_w)_{w \in W}) & = \oplus_{w \in W} j_{w!}(A_w),\\
        \alpha & : \bigoplus_{w \in W} D_U(G/B) \to D(G/B), & \alpha((\mathcal{F}_w)_{w \in W}) & = \oplus_{w \in W} \pi_w^*(\mathcal{F}_w),\\
        \beta & : D(G/B) \to \bigoplus_{w \in W} D_U(G/B), & \beta(\mathcal{F}) & = (\pi^w_* \mathcal{F})_{w \in W}.
    \end{align*}
    We use ${}^\circ j^!$ to denote the abelian version of $j^!$ considered as a functor from $\mathcal{A}_f$ to $\O$. The adjunctions in Proposition \ref{prop:originaladjunctions} imply that there are adjunctions
    \begin{align*}
        (j_!, j^!), & \qquad (\alpha, \beta).
    \end{align*}
\end{definition}

\begin{definition}
    Let $\Da$ and $\Dg$ be the cocompletions of $\Da_0$ and $\Dg_0$ in the categories $D(\gmod_0)$ and $D(G/B)$ respectively.
\end{definition}
Proposition \ref{cor:dggamma} implies that the localization functor $\Gamma : D(G/B) \to D(\gmod_0)$ induces an equivalence $\Dg \to \Da$, whcih we continue to refer to as $\Gamma$. 

\begin{proposition}\label{prop:imageco}
    $\Da$ is the internal cocompletion of the full dg subcategory of $D(\gmod_0)$ generated by the image of $j_{!}$. 
    
    Equivalently, $\Dg$ is the cocompletion of the full dg subcategory of $D(G/B)$ generated by the image of $\alpha$. 
\end{proposition}
\begin{proof}
    The localization equivalence $\Dg \to \Da$ and the equivalence of functors given in Lemma \ref{lem:functordict1} and Corollary \ref{cor:functordict2}, it is enough to prove the statement for $\Da$ alone. By Proposition \ref{prop:afsimples}, any object of $\Afmod$ is an extension of objects lying in the image of $j_!$, so $\Da_0$ is the full dg subcategory of $D(\gmod_0)$ generated by the image of $j_!$. Thus the result follows from taking cocompletions in $D(\gmod_0)$.
\end{proof}

\subsection{Verifying the Barr--Beck--Lurie conditions}

\subsubsection{A right adjoint for $j^!$}

\begin{definition}
    Let 
    \begin{align*}
        R_{\nabla} & : D_U(G/B) \to D_U(G/B)\\ 
        R_{\Delta} & : D_U(G/B) \to D_U(G/B)
    \end{align*}
    be the inverse equivalences defined for any element $M \in D(\tO)$ by 
    \begin{align*}
        R_{\nabla}(M) & = \hat{\Delta}_{w_0} * M [\ell(w_0)]\\ 
        R_{\Delta}(M) & = \hat{\nabla}_{w_0} * M [-\ell(w_0)].
    \end{align*}
\end{definition}

\begin{definition}
    For any $w \in W$, let $\gamma_w : D(\tO) \to \Dg$ be the functor defined by
    \begin{align*}
        \gamma_{w} = \pi_{w_0w}^* \circ R_{\Delta}.
    \end{align*}
\end{definition}

\begin{proposition}\label{prop:betagamma}
    For any $w \in W$, there is an adjunction $(\pi_*^w, \gamma_w)$ of functors 
    \begin{align*}
        \Dg \rightleftarrows D_U(G/B).
    \end{align*}
    This implies there is an adjunction $(\beta, \gamma)$ of functors
    \begin{align*}
        \Dg \rightleftarrows \bigoplus_{w \in W} D_U(G/B).
    \end{align*}
\end{proposition}

\begin{proof}
    We first show the adjunction between $D_U(G/B)$ and the essential image $\mathcal{D}'$ of $\pi_w^*$ in $\Dg$. We claim that if $\mathcal{G}$ belongs to this image, then there is a canonical isomorphism
    \begin{align}
        \label{eqn:duality} \pi_*^w\mathcal{G} \cong R_{\nabla} \circ \pi_!^{w_0w}\mathcal{G},
    \end{align}
    which will immediately imply the result by the adjunction $(\pi_!^{w_0w}, \gamma_{w_0w})$ along with the fact that $R_{\nabla}^{-1} \cong R_{\Delta}$. To show (\ref{eqn:duality}) on $\mathcal{D}'$, it is enough to show it for any element of the form $\pi_y^*(\mathcal{F})$ for $\mathcal{F} \in D_U(G/B)$. In fact, since $\pi_y^*(\mathcal{F}) = \sigma_{y}^*\pi_e^*(\mathcal{F})$ and 
    \begin{align*}
        \pi_*^w\sigma^*_y & \cong \pi_*^{wy}\\ 
        \pi_!^{w_0w}\sigma_y^* & \cong \pi_!^{w_0wy}
    \end{align*}
    for any $y \in W$, it suffices to prove that there is a canonical isomorphism
    \begin{align*}
        \pi_*^w\pi_e^*\mathcal{F} \cong R_{\nabla} \pi_!^{w_0w} \pi_e^*\mathcal{F}
    \end{align*}
    for any $w \in W$. We know by Proposition \ref{prop:convinterpretation} that
    \begin{align*}
        \pi_*^{w}\pi_e^*(\mathcal{F}) & \cong \hat{\nabla}_w * \mathcal{F}[-\ell(w)],\\
        \pi_!^{w_0w}\pi_e^*(\mathcal{F}) & \cong \hat{\Delta}_{w_0w} * \mathcal{F}[\ell(ww_0)],
    \end{align*}
    so the result follows from the canonical isomorphism
    \begin{align*}
        \hat{\nabla}_w[-\ell(w)] & \cong \ \hat{\nabla}_{w_0}[-\ell(w_0)]* \hat{\Delta}_{w_0w}[\ell(ww_0)].
    \end{align*}

    Now we claim that this adjunction further extends to the cocompletion of $\mathcal{D}'$ in $D(\gmod_0)$, which by Proposition \ref{prop:imageco} is $\Dg$, thereby giving an adjunction between $\Dg$ and $D_U(G/B)$. Let $I : \mathcal{D}' \to \Dg$ be the inclusion. Then the preceding computation implies that there is an $I$-relative adjunction between $\beta I$ and $\gamma = I\circ \gamma|_{\mathcal{D}'}$ in the sense of \cite{Ulmer}. Since $I$ is dense and since $\Dg$ admits colimits, $\beta$ is clearly the pointwise left extension of $\beta \circ I$ along $I$, so $\beta$ is left adjoint to $\gamma$, e.g. by \cite[Proposition 5.4.7]{Arkor}.
\end{proof}

Passing to the algebraic side under localization, we get the following corollary.
\begin{corollary}\label{cor:jlowerstar}
    The functor $j^!$ has a right adjoint which we call $j_*$. In particular, $j^!$ is both continuous and cocontinuous as a dg functor \[j^! : \Da \to \oplus_{w \in W} D(\tO).\]
\end{corollary}

\subsection{Conservativity of \texorpdfstring{$j^!$}{j!}}

\begin{lemma}\label{lem:abcons}
    Suppose that $M \in \Afmod$. Then ${}^\circ j^*M \neq 0$ and ${}^\circ j^!M \neq 0$.
\end{lemma}
\begin{proof}
    Since $M$ has finite length, it admits some simple submodule $L$ and a simple quotient $L'$. By Proposition \ref{prop:afsimples}, there exists some $w \in W$ for which $L$ is a submodule of ${}^\circ j_w^!M$ (and some $w \in W$ such that $L'$ is a quotient of ${}^\circ j_w^*M$), which in particular means that $\oplus_{w \in W} {}^\circ j_w^*(M) \neq 0$, $\oplus_{w \in W} {}^\circ j_w^!(M) \neq 0$.
\end{proof}

It is then a general fact (which follows by left or right exactness respectively) that the derived functors must too be conservative.
\begin{corollary}
    The functors $j^*$ and $j^!$ are conservative on $\Da_0$.
\end{corollary} 

We now prove conservativity of $j^!$ (resp. $\beta$) on all of $\Da$ (resp. $\Dg$).
\begin{corollary}\label{cor:conservative}
    The functor $j^! : \Da \to \oplus_{w \in W} D(\tO)$ is conservative. Equivalently, the functor $\beta : \Dg \to \oplus_{w \in W} D(\tO)$ is conservative.
\end{corollary}
\begin{proof}
    By definition, $\mathcal{D}^{\mathrm{alg}}$ is the cocompletion of $\Da_0 \subset D(\mathfrak{g}\mathrm{-mod}_{\hat{0}})$. The obvious functor $\Da_0 \to \Da$ is colimit-preserving, and therefore by the adjoint functor theorem has a right adjoint functor $\rho$. The functor $\rho$ is conservative by the correspondence between conservativity of right adjoints and generation under colimits by the corresponding left adjoint which is explained in \cite{Monadic}.
    
    We note that $(j^!)|_{\Da_0} \circ \rho : \mathcal{D}^{\mathrm{alg}} \to \oplus_{w \in W} D(\tO)$ is the composition of conservative functors and is therefore conservative; further, by adjunction we see that \[(j^!)|_{\Da_0} \circ \rho = j^!.\]
\end{proof}

\subsection{The equivalence}

\subsubsection{Statement of the main theorem}
We recall that if $\mathsf{T} : \mathcal{D} \to \mathcal{D}$ is a dg monad on some dg category $\mathcal{D}$, we denote by $\mathcal{D}^\mathsf{T}$ the category of dg algebras over $\mathsf{T}$.
\begin{prop}\label{prop:monad}
    There are canonical equivalences of categories
    \begin{align*}
        \Da & \cong \left(\bigoplus_{w \in W} D(\tO)\right)^{\mathsf{T}^{\mathrm{alg}}}\\
        \Dg & \cong \left(\bigoplus_{w \in W} D_U(G/B)\right)^{\mathsf{T}^{\mathrm{geom}}},
    \end{align*}
    where
    \begin{align*}
        \mathsf{T}^{\mathrm{alg}} & = j^!j_!\\
        \mathsf{T}^{\mathrm{geom}} & = \beta\alpha
    \end{align*}
\end{prop}
\begin{proof}
    There is an adjunction $(j_!, j^!)$. The composition $\mathsf{T}^{\mathrm{alg}} = j^!j_!$ defines a monad on $\bigoplus_{w \in W} D(\mathcal{O}^w)$. We then have that $j^!$ factors through a functor $j^! : \mathcal{D} \to (\bigoplus_{w \in W} D(\mathcal{O}^w))^{\mathsf{T}^{\mathrm{alg}}}$, the category of algebras over the monad $\mathsf{T}^{\mathrm{alg}}$.
    
    Since $j^!$ is a left adjoint by Corollary \ref{cor:jlowerstar}, it preserves colimits. So by the Barr--Beck--Lurie theorem, to show that $j^!$ is an equivalence, it is enough to show that $j^!$ is conservative, which is shown in Corollary \ref{cor:conservative}.
\end{proof}

\begin{theorem}\label{thm:maingeom}
    There is a canonical equivalence of categories \[\Dg \cong \mathrm{KL}_{\mathcal{O}}^\vee.\]
\end{theorem}
\begin{proof}
    Since $\Phi_{w} = \pi_{*}^w\pi_e^*$, by Proposition \ref{prop:klismonad} the category $\widetilde{\mathrm{KL}}$ is exactly the category of algebras over the monad $\mathsf{T}^{\mathrm{geom}}$, so the result follows from Proposition \ref{prop:monad}.
\end{proof}

\subsection{\texorpdfstring{$t$}{t}-structures under the equivalence}

Theorems \ref{thm:mainalg} and \ref{thm:maingeom} give us an equivalence between $\Da$ (or $\Dg$) and the category of $W$-tuples of objects of $D(\O)$ equipped with morphisms satisfying certain gluing compatibilities. The category $\Da$ has a natural $t$-structure given by the restriction of the tautological $t$-structure on $D(\gmod_0)$; we use $\tau_{\leq k}$ and $\tau_{\geq k}$ to denote its truncation functors and ${}^\heartsuit \Da$ to denote its heart.

We can then ask whether there is a characterization of the corresponding $t$-structure on $\KL^\vee$ under the equivalence in Theorem \ref{thm:mainalg}. 

\begin{definition}
    Let ${}_f\KL^\vee$ be the image of $\Dg_0$ (or equivalently $\Dg_0$) under the equivalence in Theorem \ref{thm:mainalg} (or Theorem \ref{thm:maingeom}).
\end{definition}
Since $\Da_0$ admits a bounded $t$-structure which is inherited from the tautological $t$-structure on $D^b(\gmod_0)$, which also agrees with the bounded $t$-structure on $\Dg_0$ inherited from the perverse $t$-structure on $D^b(G/B)$, passing this through the main equivalences ensures that ${}_f\KL^\vee$ admits a bounded $t$-structure which we denote by $({}_f^{\leq 0}\KL^\vee, {}_f^{\geq 0}\KL^\vee)$.

\begin{proposition}\label{prop:boundedt}
    The bounded $t$-structure on ${}_f\KL^\vee$ obtained abstractly from the obvious $t$-structures on $\Da_0$ or $\Dg_0$ under the equivalences in Theorem \ref{thm:mainalg} and Theorem \ref{thm:maingeom} can be described explicitly as follows:
    \begin{align}
        {}_f^{\geq 0}\KL^\vee & = \{(A_w)_{w \in W} \in {}_f\KL^\vee ~|~ A_w \in D^{\geq 0}_U(G/B) \textrm{ for all } w \in W\},\label{eqn:geqkl}\\
        {}_f^{\leq 0}\KL^\vee & = \{(A_w)_{w \in W} \in {}_f\KL^\vee ~|~ \hat{\Delta}_{w_0} * A_w \in D^{\leq \ell(w_0)}_U(G/B) \textrm{ for all } w \in W\}.\label{eqn:leqkl}
    \end{align}
\end{proposition}
\begin{proof}
    First, we will show that under the equivalence in Theorem \ref{thm:maingeom}, the image of the connective and coconnective parts of the $t$-structure on $\Dg_0$ inherited from the perverse $t$-structure on $D(G/B)$ lies in the sets described above. Indeed, suppose $\mathcal{F} \in \Dg \cap D^{\geq 0}(G/B)$. Then by left $t$-exactness of $\pi_*$, we have that 
    \[\pi_*^w\mathcal{F} = D^{\geq 0}_U(G/B)\]
    for any $w \in W$, proving the first containment. If instead $\mathcal{F} \in \Dg \cap D^{\leq 0}(G/B)$, then by the proof of Proposition \ref{prop:betagamma}, for any $w \in W$ we have
    \[\pi_*^w\mathcal{F} \cong R_\nabla \circ \pi_!^{w_0w} \mathcal{F}.\]
    Since $\mathcal{F} \in D^{\leq 0}(G/B)$, and $\pi_!^{w_0w}$ is right $t$-exact, we have $\pi_!^{w_0w} \mathcal{F} \in D^{\leq 0}(G/B)$, so
    \[\hat{\Delta}_{w_0} * \pi_*^{w}\mathcal{F} \cong \pi_!^{w_0w}\mathcal{F} [-\ell(w_0)] \in D^{\leq \ell(w_0)}_U(G/B).\]
    To show the other containment, note that in the above we have shown that if $(A_w)_{w \in W} \in {}_f\KL^\vee$ is in the image of an object $\mathcal{F} \in \Dg_0$, then lying in the set on the right-hand side of (\ref{eqn:geqkl}) (resp. (\ref{eqn:leqkl})) is equivalent to knowing that $\pi_*^w\mathcal{F} \in D^{\geq 0}_U(G/B)$ for all $w$ (resp. $\pi_!^w\mathcal{F} \in D^{\leq 0}_U(G/B)$ for all $w$). Suppose this holds, then it remains to show that $\mathcal{F} \in D^{\geq 0}(G/B)$ (resp. $\mathcal{F} \in D^{\leq 0}(G/B)$).

    If instead $\mathcal{F} \not\in D^{\geq 0}_{U}(G/B)$, then because $\mathcal{F} \in D^b(G/B)$, there exists some minimal $k \in \mathbb{Z}_{<0}$ such that ${}^pH^k(\mathcal{F}) \neq 0$. For any $w \in W$, there are then distinguished triangles
    \[\begin{tikzcd}
        ({}^pH^k\mathcal{F}) \arrow[r] & \mathcal{F}[k] \arrow[r] & \tau_{\geq k+1}\mathcal{F}[k] \arrow[r, "{[1]}"] & ~\\
        \pi_*^w({}^pH^k\mathcal{F}) \arrow[r] & \pi_*^w(\mathcal{F})[k] \arrow[r] & \pi_*^w(\tau_{\geq k+1}\mathcal{F})[k], \arrow[r, "{[1]}"] & ~
    \end{tikzcd}\]
    and the corresponding long-exact sequence gives that
    \[\begin{tikzcd}
        0 \arrow[r] & {}^pH^0(\pi_*^w({}^pH^k\mathcal{F})) \arrow[r] & {}^pH^k(\pi_*^w(\mathcal{F})) \arrow[r] & \dots
    \end{tikzcd}\]
    is exact. By Lemma \ref{lem:abcons} and Corollary \ref{cor:functordict2}, there must be some $w \in W$ such that ${}^pH^0(\pi_*^w({}^pH^k\mathcal{F})) \neq 0$ and therefore ${}^pH^k(\pi_*^w(\mathcal{F})) \neq 0$, contradicting the assumption that $\pi_*^w \mathcal{F} \in D^{\geq 0}(G/B)$ for all $w \in W$. Applying the same argument to $\pi_!^w$ gives the same result to show that $\mathcal{F} \in D^{\leq 0}(G/B)$ if one first assumes that $(\pi_*^w\mathcal{F})$ lies in the right-hand side of (\ref{eqn:leqkl}).
\end{proof}

\begin{theorem}\label{thm:tstruct}
    The natural $t$-structure 
    \[({}^{\leq 0}\Dg, {}^{\geq 0}\Dg)\]
    on $\Dg$ inherited from the perverse $t$-structure on $D(G/B)$ gives rise to a well-defined $t$-structure 
    \[({}^{\leq 0}\KL^\vee, {}^{\geq 0}\KL^\vee)\]
    on $\KL^\vee$ under the equivalence in Theorem \ref{thm:mainalg}. Further, its coconnective part has a particularly direct description, as
    \begin{align}
        {}^{\geq 0}\KL^\vee & = \{(A_w)_{w \in W} \in \KL^\vee ~|~ A_w \in D^{\geq 0}_U(G/B) \textrm{ for all } w \in W\},\label{eqn:geqklreal}.
    \end{align}
\end{theorem}
\begin{proof}
    Suppose that $\mathcal{F} \in {}^{\geq 0}\Dg = \Dg \cap D^{\geq 0}(G/B)$. Then $\mathcal{F} \cong \mathrm{colim}( \mathcal{F}_i)$ for $\mathcal{F}_i \in \Dg_0$. Since $\tau_{\geq 0}$ is a left adjoint, it commutes with colimits, and so $\mathcal{F} \cong \mathrm{colim}(\tau_{\geq 0}\mathcal{F}_i)$. This means ${}^{\geq 0}\Dg$ is exactly the subset of $\Dg$ generated under colimits by ${}^{\geq 0}\Dg_0$. By Proposition \ref{prop:boundedt}, this is exactly the subset of $\Dg$ generated under colimits by the right-hand side of (\ref{eqn:geqkl}), which clearly coincides with the right-hand side of (\ref{eqn:geqklreal}) for the same reason; this completes the proof.
\end{proof}
\begin{remark}
    We note that Proposition \ref{prop:boundedt} cannot be naively generalized to give a direct description of ${}^{\leq 0}\KL^\vee$ as 
    \[{}^{\leq 0}\KL^\vee = \{(A_w)_{w \in W} \in \KL^\vee ~|~ \hat{\Delta}_{w_0} * A_w \in D^{\leq \ell(w_0)}_U(G/B) \textrm{ for all } w \in W\}.\]
    For one, the proof of Theorem \ref{thm:tstruct} cannot be dualized as $\Dg$ is defined in terms of colimits rather than limits. Further, we now explain that even in the case of $\mathrm{SL}_2$, the dual result is actually false.

    Indeed, consider the object $(T_s, 0)$, and note that
    \begin{align}
        \hat{\Delta}_s * T_s \cong T_s \in D^{\leq 0}_U(G/B),
    \end{align}
    which would imply that $(T_s, 0) \in {}^{\leq -1}\KL^\vee$ if such a description were true. However, we also clearly have $(T_s, 0) \in {}^{\geq 0}\KL^\vee$, which would be a contradiction.

    This also serves as a proof that $(T_s, 0)$ is an example of an object in $\KL^\vee$ which does not lie in the image of $\Da_0$ or $\Dg_0$ under our main equivalence.
\end{remark}

\section{Explicit objects}\label{sec:explicit}

Although we used abstract machinery to establish the equivalences in Theorems \ref{thm:mainalg} and \ref{thm:maingeom}, in some cases it is possible to explicitly identify the images of objects under these equivalences. In this section, we work out some of these images explicitly in the case of $\mathfrak{g} = \mathfrak{sl}_2$. We begin explaining an important class of objects for general $\mathfrak{g}$ and later in the section we specialize to the rank $1$ case. This will serve as an explicit $\mathfrak{sl}_2$-blueprint for the important class of examples we study in Section \ref{sec:coherent}.

\subsection{Tilting objects in \texorpdfstring{$\KL^\vee$}{the Koszul dual Kazhdan--Laumon category}}

Recall that by \cite{KLCatO}, the simple objects of $\KL$ are in bijection with pairs $(w, zP(w))$, where $w \in W$ and $zP(w)$ is a coset in $W/P(w)$, where $P(w)$ is defined in Corollary \ref{cor:countsimples}. Any such simple can be described as a tuple with $L_w$ in the entries lying in the coset $zP(w)$ and $0$ elsewhere, where $L_w \in \O$ is the simple object with highest weight $w \cdot (-2\rho)$, or equivalently in the geometric setup it is the simple perverse sheaf $\mathrm{IC}_w$ corresponding to the stratum $BwB \subset G/B$.

Following the Koszul duality philosophy of Section \ref{subsec:kdkl} applied to this classification of simple objects, we then get the following result and definition.

\begin{lemma}
    For any $w \in W$ and coset $zP(w) \in W/P(w)$, there exists an object $T_w^z \in \KL^\vee$ such that for any $y \in W$,
    \begin{align}
        j_y^*(T_w^z)& = \begin{cases}
        T_w & y \in zP(w),\\
        0 & y \not\in zP(w),
        \end{cases}
    \end{align}
    where $T_w$ is the indecomposable tilting object of $\O$ indexed by $w \in W$  (equivalently, the unique indecomposable tilting perverse sheaf in $\mathrm{Perv}_U(G/B)$ obtained by tilting extension from the stratum $UwB$). We call these \emph{tilting objects} in $\KL^\vee$.
\end{lemma}

\subsection{Geometric examples for \texorpdfstring{$\mathfrak{sl}_2$}{sl2}}\label{sec:geomexamples}

For the rest of this section, we now specialize to the case of $\mathfrak{g} = \mathfrak{sl}_2$, writing $W = \{e, s\}$. In this case $G/B \cong \mathbb{P}^1$, so we let
\begin{align*}
    i_0 & : \{0\} \hookrightarrow \mathbb{P}^1 \\
    i_\infty & : \{\infty\} \hookrightarrow \mathbb{P}^1\\
    j^0 & : \mathbb{P}^1 \setminus \{0\} \hookrightarrow \mathbb{P}^1,\\
    j^{\infty} & : \mathbb{P}^1 \setminus \{\infty\} \hookrightarrow \mathbb{P}^1,\\
    j^{0,\infty} & : \mathbb{P}^1 \setminus \{0,\infty\} \hookrightarrow \mathbb{P}^1.
\end{align*}
be the inclusion maps. We choose a Borel subgroup $B$ and therefore a unipotent radical $U$ with conventions such that $0$ is a $U$-orbit and $\infty$ is a $U^-$-orbit. Objects in the images of the functors
\begin{align*}
    i_{0!}, j^0_!, j^0_*
\end{align*}
are therefore automatically $U$-equivariant whereas objects in the images of the functors \begin{align*}
    i_{\infty!}, j^\infty_!, j^\infty_*
\end{align*}
are $U^-$-equivariant.

We note that objects in $\Dg$ which lie in the image of $\pi_w^*$ for some $w \in W$ are easy to describe as elements of $\KL^\vee$ because of the direct descriptions of the right and left adjoints to restrictions on $\KL$ and therefore $\KL^\vee$. From these descriptions we see that, under the equivalence $\Dg \cong \KL^\vee$, we have
\begin{align*}
    i_{0!}\mathbb{C} & \longleftrightarrow (\Delta_e, \Delta_s[1]),\\
    i_{\infty!}\mathbb{C} & \longleftrightarrow (\Delta_s[1], \Delta_e),\\
    j^0_*\mathbb{C}[1] & \longleftrightarrow (\nabla_s, \Delta_e[1])\\
    j^\infty_*\mathbb{C}[1] & \longleftrightarrow (\Delta_e[1], \nabla_s).
\end{align*}

We now consider extensions of a certain nontrivial local system on $\mathbb{P}^1\setminus \{0, \infty\}$. To do this, we make a more general definition for any torus $T$.
\begin{definition}\label{def:locsys}
    Let $\tilde{\mathcal{L}}$ be the ``ind-universal local system" on $T$ corresponding to the $\mathbb{C}[\mathfrak{h}]$-module $\mathbb{C}(\mathfrak{h})/\mathbb{C}[\mathfrak{h}]$.

    We also make a local version of this definition, letting $\tilde{\mathcal{L}}^{\mathrm{unip}}$ be the \emph{ind-unipotent universal local system} corresponding to the $\mathbb{C}[\mathfrak{h}]$-module $\mathbb{C}((\mathfrak{h}))/\mathbb{C}[[\mathfrak{h}]]$.
\end{definition}

\begin{remark}
    We note that in the literature, the ``pro-universal local system" corresponding to the $\mathbb{C}[\mathfrak{h}]$-module $\mathbb{C}[\mathfrak{h}]$ and its unipotent analog corresponding to $\mathbb{C}[[\mathfrak{h}]]$ are much more ubiquitous, and are simply referred to as the universal (unipotent) local system. Since in the current setup we wish to work with colimits instead of limits, the local systems in Definition \ref{def:locsys} will be more natural from our point-of-view.
\end{remark}

We now consider the object
\begin{align*}
    j_*^{0,\infty}\mathcal{L}[1]
\end{align*}
where $\mathcal{L}$ is the universal local sytem on $\mathbb{G}_m \cong \mathbb{P}^1 \setminus \{0, \infty\}$, which is neither $U$- nor $U^-$-equivariant. It also does not lie in the category $\mathrm{Perv}_U(G/B)$, but we will see that it does lie in (the perverse heart of) the completion $\Dg$. Further, one can define an object $\mathcal{F}_{0}$ as $j_{!}^\infty$ applied to the $*$-extension of $\mathcal{L}[1]$ along the inclusion $\mathbb{P}^1 \setminus \{0, \infty\}$ to $\mathbb{P}^1 \setminus \{\infty\}$. Swapping the roles of $0$ and $\infty$ in this construction gives us an additional object $\mathcal{F}_{\infty}$.

One can then check that under the equivalence $\Dg \cong \KL^\vee$, we have
\begin{align*}
    j_*^{0,\infty} \mathcal{L}[1] & \longleftrightarrow (\nabla_s, \nabla_s),\\
    \mathcal{F}_0 & \longleftrightarrow (T_s, 0),\\
    \mathcal{F}_\infty & \longleftrightarrow (0, T_s),\\
\end{align*}
where $T_s$ is the indecomposable tilting sheaf on $\mathbb{P}^1$ obtained as the unique nontrivial extension
\[\begin{tikzcd}
    0 \arrow[r] & \Delta_s \arrow[r] & T_s \arrow[r] & \Delta_e \arrow[r] & 0.
\end{tikzcd}\]
This serves as a set of examples of nontrivial objects of $\KL^\vee$ whose tuple entries are each perverse but which do not lie in the natural inclusion of any individual $w$-conjugate of $\mathrm{Perv}_U(G/B)$.

\subsection{Algebraic examples for \texorpdfstring{$\mathfrak{sl}_2$}{sl2}}

We now carry these geometric examples forward to the algebraic setting and explain how to interpret them as elements of $\Da$.

We remain in the case $\mathfrak{g} = \mathfrak{sl}_2$, so $\mathcal{U}(\mathfrak{h}) \cong \mathbb{C}[h]$. For any $\lambda \in \mathbb{C}$, we can consider the $\mathbb{C}[h]$-module $H_\lambda = \mathbb{C}((h - \lambda))/\mathbb{C}[[h - \lambda]]$, which is a completion of the space spanned as a $\mathbb{C}$-vector space by vectors $\{v_i\}_{i \geq 0}$ such that $(h - \lambda)v_i = v_{i-1}$ for $i \geq 1$ and $(h - \lambda)v_0 = 0$.

We can then consider the following objects in $\gmod_0$:
\begin{align}
    M_0 & = \mathrm{coInd}_{\mathcal{U}(\mathfrak{h})}^{\mathcal{U}(\mathfrak{g})_0} H_0,\\
    M_{-2} & = \mathrm{coInd}_{\mathcal{U}(\mathfrak{h})}^{\mathcal{U}(\mathfrak{g})_0} H_{-2},\\
    M_2 & = \mathrm{coInd}_{\mathcal{U}(\mathfrak{h})}^{\mathcal{U}(\mathfrak{g})_0} H_2.
\end{align}

One can then compare these modules to the diagrams in Figures \ref{fig:m0}, \ref{fig:m-2}, and \ref{fig:m2}, which are drawn as follows. Dots above an integer $i$ represent a vector of weight $i$ in the representation $M_\lambda$. Rightward arrows represent the action of $e$ while leftward arrows represent the action of $f$. For simplicity, nodes are placed in a vertical hierarchy and arrows are drawn only ``up to lower nodes." For example, a rightward facing arrow from nodes which represent vectors $v_1$ and $v_2$ of weight $k$ and $k + 2$ respectively indicates that $e(v_1)$ is a linear combination of $v_2$ along with vectors corresponding to nodes above $k + 2$ which lie below the node corresponding to $v_2$.

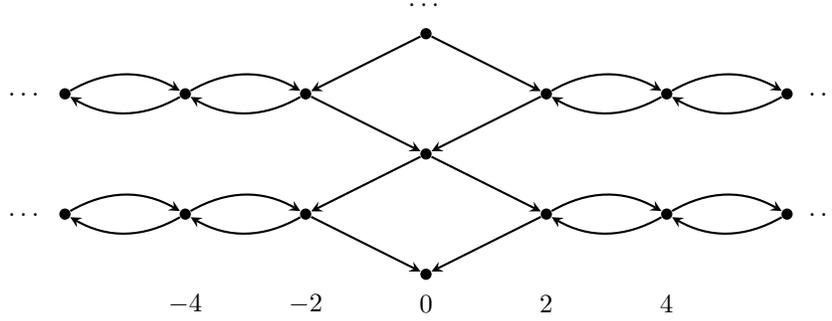
\begin{figure}
    \centering
    \begin{tikzpicture}[scale=0.8, every node/.style={circle, fill=black, inner sep=1.5pt}, >=stealth]
            \node[label=above:{$\dots$}] (w0) at (0,0) {};
            \foreach \x in {-6, -4,-2,2,4, 6}
            \node (w\x) at (\x,-1) {};
            \node (lw0) at (0,-2) {};
            \foreach \x in {-6, -4,-2,2,4, 6}
            \node (lw\x) at (\x,-3) {};
            \node (llw0) at (0,-4) {};
            \foreach \x in {-4,-2,0,2,4}
            \node[fill=none] at (\x,-4.5) {$\x$};
            \node[fill=none] at (-6.65,-1) {$\dots$};
            \node[fill=none] at (6.65,-1) {$\dots$};
            \node[fill=none] at (-6.65,-3) {$\dots$};
            \node[fill=none] at (6.65,-3) {$\dots$};

            \draw[->, thick, black, bend right=0] (w0) to node[above, sloped, midway, fill=none, inner sep=0pt, outer sep=0pt] {} (w-2);
            \draw[->, thick, black, bend left=30] (w-2) to node[above, sloped, midway, fill=none, inner sep=0pt, outer sep=0pt] {} (w-4);
            \draw[->, thick, black, bend right=-30] (w-4) to node[above, sloped, midway, fill=none, inner sep=0pt, outer sep=0pt] {} (w-2);
            \draw[->, thick, black, bend right=-30] (w-4) to node[above, sloped, midway, fill=none, inner sep=0pt, outer sep=0pt] {} (w-6);
            \draw[->, thick, black, bend right=-30] (w-6) to node[above, sloped, midway, fill=none, inner sep=0pt, outer sep=0pt] {} (w-4);
            
            \draw[->, thick, black, bend left=0] (w0) to node[above, sloped, midway, fill=none, inner sep=0pt, outer sep=0pt] {} (w2);
            \draw[->, thick, black, bend right=-30] (w2) to node[above, sloped, midway, fill=none, inner sep=0pt, outer sep=0pt] {} (w4);
            \draw[->, thick, black, bend left=30] (w4) to node[above, sloped, midway, fill=none, inner sep=0pt, outer sep=0pt] {} (w2);
            \draw[->, thick, black, bend left=30] (w4) to node[above, sloped, midway, fill=none, inner sep=0pt, outer sep=0pt] {} (w6);
            \draw[->, thick, black, bend left=30] (w6) to node[above, sloped, midway, fill=none, inner sep=0pt, outer sep=0pt] {} (w4);

            \draw[->, thick, black, bend right=0] (lw0) to node[above, sloped, midway, fill=none, inner sep=0pt, outer sep=0pt] {} (lw-2);
            \draw[->, thick, black, bend left=30] (lw-2) to node[above, sloped, midway, fill=none, inner sep=0pt, outer sep=0pt] {} (lw-4);
            \draw[->, thick, black, bend right=-30] (lw-4) to node[above, sloped, midway, fill=none, inner sep=0pt, outer sep=0pt] {} (lw-2);
            \draw[->, thick, black, bend right=-30] (lw-4) to node[above, sloped, midway, fill=none, inner sep=0pt, outer sep=0pt] {} (lw-6);
            \draw[->, thick, black, bend right=-30] (lw-6) to node[above, sloped, midway, fill=none, inner sep=0pt, outer sep=0pt] {} (lw-4);
            
            \draw[->, thick, black, bend left=0] (lw0) to node[above, sloped, midway, fill=none, inner sep=0pt, outer sep=0pt] {} (lw2);
            \draw[->, thick, black, bend right=-30] (lw2) to node[above, sloped, midway, fill=none, inner sep=0pt, outer sep=0pt] {} (lw4);
            \draw[->, thick, black, bend left=30] (lw4) to node[above, sloped, midway, fill=none, inner sep=0pt, outer sep=0pt] {} (lw2);
            \draw[->, thick, black, bend left=30] (lw4) to node[above, sloped, midway, fill=none, inner sep=0pt, outer sep=0pt] {} (lw6);
            \draw[->, thick, black, bend left=30] (lw6) to node[above, sloped, midway, fill=none, inner sep=0pt, outer sep=0pt] {} (lw4);

            \draw[->, thick, black, bend left=0] (w-2) to node[above, sloped, midway, fill=none, inner sep=0pt, outer sep=0pt] {} (lw0);
            \draw[->, thick, black, bend left=0] (w2) to node[above, sloped, midway, fill=none, inner sep=0pt, outer sep=0pt] {} (lw0);

            \draw[->, thick, black, bend left=0] (lw2) to node[above, sloped, midway, fill=none, inner sep=0pt, outer sep=0pt] {} (llw0);
            \draw[->, thick, black, bend left=0] (lw-2) to node[above, sloped, midway, fill=none, inner sep=0pt, outer sep=0pt] {} (llw0);
        \end{tikzpicture}
    \caption{A schematic picture of the module $M_0$.\label{fig:m0}}
\end{figure}

\begin{figure}
    \centering
    \begin{tikzpicture}[scale=0.8, every node/.style={circle, fill=black, inner sep=1.5pt}, >=stealth]

                    \node[label=above:{$\dots$}] (w0) at (0,0) {};
                    \foreach \x in {-6, -4,-2,2,4, 6}
                    \node (w\x) at (\x,-1) {};
                    \node (lw0) at (0,-2) {};
                    \foreach \x in {-6, -4,-2}
                    \node (lw\x) at (\x,-3) {};

                    \foreach \x in {-4,-2,0,2,4}
                    \node[fill=none] at (\x,-3.7) {$\x$};

                    \node[fill=none] at (-6.65,-1) {$\dots$};
                    \node[fill=none] at (6.65,-1) {$\dots$};
                    \node[fill=none] at (-6.65,-3) {$\dots$};

                    \draw[->, thick, black, bend right=0] (w0) to node[above, sloped, midway, fill=none, inner sep=0pt, outer sep=0pt] {} (w-2);
                    \draw[->, thick, black, bend left=30] (w-2) to node[above, sloped, midway, fill=none, inner sep=0pt, outer sep=0pt] {} (w-4);
                    \draw[->, thick, black, bend right=-30] (w-4) to node[above, sloped, midway, fill=none, inner sep=0pt, outer sep=0pt] {} (w-2);
                    \draw[->, thick, black, bend right=-30] (w-4) to node[above, sloped, midway, fill=none, inner sep=0pt, outer sep=0pt] {} (w-6);
                    \draw[->, thick, black, bend right=-30] (w-6) to node[above, sloped, midway, fill=none, inner sep=0pt, outer sep=0pt] {} (w-4);
                    
                    \draw[->, thick, black, bend left=0] (w0) to node[above, sloped, midway, fill=none, inner sep=0pt, outer sep=0pt] {} (w2);
                    \draw[->, thick, black, bend right=-30] (w2) to node[above, sloped, midway, fill=none, inner sep=0pt, outer sep=0pt] {} (w4);
                    \draw[->, thick, black, bend left=30] (w4) to node[above, sloped, midway, fill=none, inner sep=0pt, outer sep=0pt] {} (w2);
                    \draw[->, thick, black, bend left=30] (w4) to node[above, sloped, midway, fill=none, inner sep=0pt, outer sep=0pt] {} (w6);
                    \draw[->, thick, black, bend left=30] (w6) to node[above, sloped, midway, fill=none, inner sep=0pt, outer sep=0pt] {} (w4);

                    \draw[->, thick, black, bend right=0] (lw0) to node[above, sloped, midway, fill=none, inner sep=0pt, outer sep=0pt] {} (lw-2);
                    \draw[->, thick, black, bend left=30] (lw-2) to node[above, sloped, midway, fill=none, inner sep=0pt, outer sep=0pt] {} (lw-4);
                    \draw[->, thick, black, bend right=-30] (lw-4) to node[above, sloped, midway, fill=none, inner sep=0pt, outer sep=0pt] {} (lw-2);
                    \draw[->, thick, black, bend right=-30] (lw-4) to node[above, sloped, midway, fill=none, inner sep=0pt, outer sep=0pt] {} (lw-6);
                    \draw[->, thick, black, bend right=-30] (lw-6) to node[above, sloped, midway, fill=none, inner sep=0pt, outer sep=0pt] {} (lw-4);

                    \draw[->, thick, black, bend left=0] (w-2) to node[above, sloped, midway, fill=none, inner sep=0pt, outer sep=0pt] {} (lw0);
                    \draw[->, thick, black, bend left=0] (w2) to node[above, sloped, midway, fill=none, inner sep=0pt, outer sep=0pt] {} (lw0);

                \end{tikzpicture}
    \caption{A schematic picture of the module $M_{-2}$.}
    \label{fig:m-2}
\end{figure}
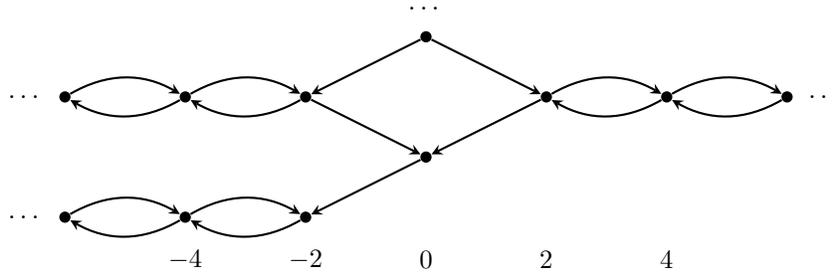

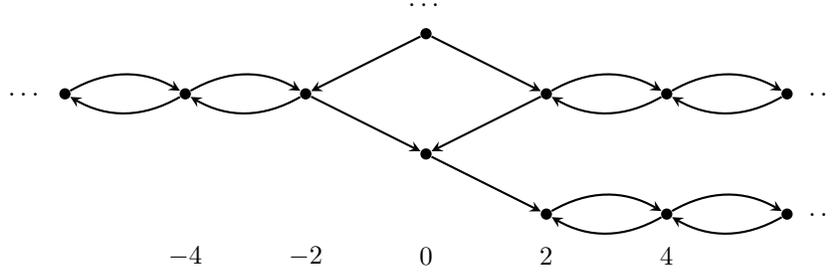
\begin{figure}
    \centering
    \begin{tikzpicture}[scale=0.8, every node/.style={circle, fill=black, inner sep=1.5pt}, >=stealth]

                    \node[label=above:{$\dots$}] (w0) at (0,0) {};
                    \foreach \x in {-6, -4,-2,2,4, 6}
                    \node (w\x) at (\x,-1) {};
                    \node (lw0) at (0,-2) {};
                    \foreach \x in {2,4, 6}
                    \node (lw\x) at (\x,-3) {};

                    \foreach \x in {-4,-2,0,2,4}
                    \node[fill=none] at (\x,-3.7) {$\x$};

                    \node[fill=none] at (-6.65,-1) {$\dots$};
                    \node[fill=none] at (6.65,-1) {$\dots$};

                    \node[fill=none] at (6.65,-3) {$\dots$};

                    \draw[->, thick, black, bend right=0] (w0) to node[above, sloped, midway, fill=none, inner sep=0pt, outer sep=0pt] {} (w-2);
                    \draw[->, thick, black, bend left=30] (w-2) to node[above, sloped, midway, fill=none, inner sep=0pt, outer sep=0pt] {} (w-4);
                    \draw[->, thick, black, bend right=-30] (w-4) to node[above, sloped, midway, fill=none, inner sep=0pt, outer sep=0pt] {} (w-2);
                    \draw[->, thick, black, bend right=0] (w-2) to node[above, sloped, midway, fill=none, inner sep=0pt, outer sep=0pt] {} (lw0);
                    \draw[->, thick, black, bend right=-30] (w-4) to node[above, sloped, midway, fill=none, inner sep=0pt, outer sep=0pt] {} (w-6);
                    \draw[->, thick, black, bend right=-30] (w-6) to node[above, sloped, midway, fill=none, inner sep=0pt, outer sep=0pt] {} (w-4);
                    
                    \draw[->, thick, black, bend left=0] (w0) to node[above, sloped, midway, fill=none, inner sep=0pt, outer sep=0pt] {} (w2);
                    \draw[->, thick, black, bend right=-30] (w2) to node[above, sloped, midway, fill=none, inner sep=0pt, outer sep=0pt] {} (w4);
                    \draw[->, thick, black, bend left=30] (w4) to node[above, sloped, midway, fill=none, inner sep=0pt, outer sep=0pt] {} (w2);
                    \draw[->, thick, black, bend left=30] (w4) to node[above, sloped, midway, fill=none, inner sep=0pt, outer sep=0pt] {} (w6);
                    \draw[->, thick, black, bend left=30] (w6) to node[above, sloped, midway, fill=none, inner sep=0pt, outer sep=0pt] {} (w4);

                    \draw[->, thick, black, bend right=0] (lw0) to node[above, sloped, midway, fill=none, inner sep=0pt, outer sep=0pt] {} (lw2);
                    \draw[->, thick, black, bend right=-30] (lw2) to node[above, sloped, midway, fill=none, inner sep=0pt, outer sep=0pt] {} (lw4);
                    \draw[->, thick, black, bend left=30] (lw4) to node[above, sloped, midway, fill=none, inner sep=0pt, outer sep=0pt] {} (lw2);
                    \draw[->, thick, black, bend left=30] (lw4) to node[above, sloped, midway, fill=none, inner sep=0pt, outer sep=0pt] {} (lw6);
                    \draw[->, thick, black, bend left=30] (lw6) to node[above, sloped, midway, fill=none, inner sep=0pt, outer sep=0pt] {} (lw4);

                    \draw[->, thick, black, bend left=0] (w2) to node[above, sloped, midway, fill=none, inner sep=0pt, outer sep=0pt] {} (lw0);

                \end{tikzpicture}
    \caption{A schematic picture of the module $M_2$.}
    \label{fig:m2}
\end{figure}
Comparing the images of the algebraic and geometric restriction functors on these modules with Section \ref{sec:geomexamples}, we get the following.
\begin{prop}
    Under the equivalences
    \begin{align*}
        \Da \cong \Dg \cong \KL^\vee, 
    \end{align*}
    we have
    \[\begin{array}{ccccc}
        M_0 & \mapsto & j_{*}^{0,\infty}\mathcal{L}[1] & \mapsto & (\nabla_s, \nabla_s),\\
        M_{-2} & \mapsto & \mathcal{F}_0 & \mapsto & (T_s, 0)\\
        M_{2} & \mapsto & \mathcal{F}_\infty & \mapsto & (0, T_s).\\
    \end{array}\]
\end{prop}

\section{Coherent families}\label{sec:coherent}

We now introduce the notion of a \emph{coherent family}, which is one of the key tools used in \cite{Fernando} and \cite{Mathieu} to give an algebraic classification of weight modules with finite-dimensional weight spaces.  We will give a natural geometric realization of all coherent families with trivial central character, which are explicitly classified in \cite{Mathieu}.

\begin{definition}\label{def:coherentfamily}
    A \emph{coherent family} of degree $d$ is a weight $\mathfrak{g}$-module $\mathcal{M}$ such that
    \begin{itemize}
        \item $\dim \mathcal{M}_\lambda = d$ for any $\lambda \in \mathfrak{h}^*$.
        \item For any $u \in \mathcal{U}(\mathfrak{g})^T$, the function $\lambda \in \mathfrak{h}^* \mapsto \mathrm{Tr}(u|_{\mathcal{M}_\lambda})$ is polynomial in $\lambda$.
    \end{itemize}
    A coherent family is called irreducible if there exists some $\lambda \in \mathfrak{h}^*$ such that $\mathcal{M}_{\lambda}$ is irreducible. For any coherent family $\mathcal{M}$ and any parameter $t \in T^*/Q$, we let $\mathcal{M}[t] \subset \mathcal{M}$ denote the submodule $\oplus_{\lambda \in t} \mathcal{M}_\lambda$.
\end{definition}

The following theorem follows from applying the classification in \cite[\S 8]{Mathieu} to the case of weight modules with trivial central character.
\begin{theorem}[\cite{Mathieu}]\label{thm:mathieucoherent}
    There are exactly $n$ distinct isomorphism classes of irreducible semi-simple coherent families for $\mathcal{U}(\mathfrak{\mathfrak{sl}_{n+1}})_0$. For $\mathfrak{g}$ outside of Type $A$, no coherent families for $\mathcal{U}(\mathfrak{g})_0$ exist.
\end{theorem}

\begin{remark}
    In \cite{Mathieu}, coherent families for $\mathcal{U}(\mathfrak{g})$ are classified, and it is shown they only exist in Types $A$ and $C$. To reduce Mathieu's classification to Theorem \ref{thm:mathieucoherent}, one checks that only $n$ of the explicit coherent families described in Type $A$ have trivial central character, while none of the Type $C$ families do.
\end{remark}

To study coherent families with trivial central character, we can then focus solely on the Type $A$ case. Thus for the remainder of this subsection, suppose $\mathfrak{g} = \mathfrak{sl}_{n+1}$; it will be convenient to set $G = PGL_{n+1}$. 

\begin{remark}
    Suppose $\mathcal{M}$ is one of the $n$ irreducible semi-simple coherent families for $\mathfrak{sl}_{n+1}$. A careful analysis of the classification in \cite[Lemma 8.3]{Mathieu} for the trivial central character $\chi = 0$ reveals the following description of $\mathcal{M}[0]$.

    Recall from \cite{KL1} that the two-sided Kazhdan--Lusztig cells for $\mathfrak{sl}_{n+1}$ correspond to partitions of $n+1$ and left cells therein are parametrized by standard Young tableaux with shape corresponding to the chosen partition. This means there is a unique two-sided cell (the minimal one among cells containing more than one element) containing $n$ left cells $c_1, \dots, c_n$. For example, when $n = 3$, three left cells in a given two-sided cell are as pictured.
    \[
\begin{tikzpicture}[every node/.style={font=\small}]

  \def\W{2.6}   
  \def\H{2.6}   
  \def\gap{0.8} 

  \coordinate (R1) at (0,0);
  \coordinate (R2) at (\W+\gap,0);
  \coordinate (R3) at (2*\W+2*\gap,0);

  \draw[thick] (R1) rectangle ++(\W,-\H);
  \draw[thick] (R2) rectangle ++(\W,-\H);
  \draw[thick] (R3) rectangle ++(\W,-\H);

  \foreach \i [count=\j] in {0.7,1.3,1.9} {
    \node at ($(R1)+(0.5*\W,-\i)$) (L\j) {};
    \node at ($(R2)+(0.5*\W,-\i)$) (M\j) {};
    \node at ($(R3)+(0.5*\W,-\i)$) (N\j) {};
  }

  \node at (L1) {$s_3s_2s_3$};
  \node at (L2) {$s_1s_3s_2s_3$};
  \node at (L3) {$s_2s_1s_3s_2s_3$};

  \node at (M1) {$s_3s_2s_3s_1$};
  \node at (M2) {$s_3s_2s_1s_2s_3$};
  \node at (M3) {$s_2s_1s_2s_3$};

  \node at (N1) {$s_3s_2s_3s_1s_2$};
  \node at (N2) {$s_3s_2s_1s_2$};
  \node at (N3) {$s_1s_2s_1$};

\end{tikzpicture}
    \]

    Mathieu's results then tell us that irreducible semi-simple coherent families can be labelled by the cells $c_1, \dots, c_n$ in such a way that the coherent family $\mathcal{M}(c_i)$ has the property that $\mathcal{M}(c_i)[0]$ is an extension of the irreducible modules $L(w\cdot (-2\rho))$ for $w \in c_i$, along with some number of copies of the trivial module $L(0)$. These are the families we will construct geometrically, guided by this fact about Kazhdan--Lusztig cells.
\end{remark}

\subsection{A geometric construction of coherent families for \texorpdfstring{$\mathfrak{sl}_{n+1}$}{sl(n+1)}}\label{sec:geomconstruction}
For the purposes of the following construction, we define an intermediate notion of a \emph{coherent complex} as follows. It mimics Mathieu's definition of coherent family while allowing for complexes not necessarily concentrated in degree zero, and will only be used as a technical tool in what follows.
\begin{definition}
     A \emph{coherent complex} is a finite complex $(\mathcal{M}^i)_{i \in \mathbb{Z}}$ of weight modules over $\mathcal{U}(\mathfrak{g})$ such that for any $u \in \mathcal{U}(\mathfrak{g})^T$, the function
    \begin{align*}
        \lambda \mapsto \sum_{i \in \mathbb{Z}} (-1)^i\mathrm{Tr}(u|_{\mathcal{M}^i_\lambda})
    \end{align*}
    is polynomial in $\lambda$, and such that the virtual dimension
    \begin{align*}
        \sum_{i \in \mathbb{Z}} (-1)^i\dim \mathcal{M}_\lambda^i
    \end{align*}
    is uniformly bounded in $\lambda$.
\end{definition}
Evidently, by this definition a coherent complex concentrated in degree $0$ is exactly a coherent family. We can thus consider coherent families themselves as a certain class of coherent complexes. Given a collection of $\{\mathcal{F}^t\}_{t \in \mathfrak{h}^*/Q} \subset D^b(G/B)$, we say it \emph{assembles to a coherent complex} (resp. coherent family) if the direct sum over all $t \in \mathfrak{h}^*/Q$ of the corresponding $\mathcal{U}(\mathfrak{g})$-modules is a coherent complex (resp. coherent family).

Given any $t \in T^* \cong \mathfrak{h}^* / Q$, let $\mathcal{L}_{t}$ be the corresponding irreducible local system on $T$. Note that $G \cong SL_{n+1}$ contains $n$ maximal proper parabolic subgroups $P_i^{\mathrm{max}}$. We can choose $P_1^{\mathrm{max}}$ with corresponding Levi subgroup isomorphic to $PGL_n$, so that $G/P_1^{\mathrm{max}} \cong \mathbb{P}^n$. For any $1 \leq i \leq n$, we let $p_i : G/B \to G/P_i^{\mathrm{max}}$ be the natural projection. For simplicity, we write
\[P = P^{\mathrm{max}}_1.\]
The variety $G/P \cong \mathbb{P}^n$ then inherits an action of $G$. Under this action, $\mathbb{P}^n$ admits an open $T$-orbit isomorphic to $T$ itself. 

\begin{definition}\label{def:g1f1}
    Let $\mathcal{G}^t_1$ be the $*$-extension of the local system $\mathcal{L}_t$ from the open $T$-orbit on $\mathbb{P}^n$ to $\mathbb{P}^n$ itself, and let $\mathcal{F}_1^t \cong p_1^*(\mathcal{G}^t_1)$.
\end{definition}
We can then interpret $\mathcal{F}_1^t$ algebraically as follows.
\begin{definition}
    Let $\mathcal{M}_1[t]$ be the $\mathfrak{g}$-module corresponding to the perverse sheaf $\mathcal{F}_1^t$ under the Beilinson--Bernstein localization functor. Let \[\mathcal{M}_1 = \bigoplus_{t \in T^*} \mathcal{M}_1[t].\]
\end{definition}

\begin{prop}\label{prop:m1}
    The module $\mathcal{M}_1$ is a coherent family of degree $1$ for $\mathfrak{sl}_{n+1}$.
\end{prop}
\begin{proof}
    We first work in the setting of $\mathcal{D}$-modules to concretely identify the $\mathcal{U}(\mathfrak{sl}_{n+1})$-module corresponding to $\mathcal{M}_1[t]$ for any fixed choice of $t$. Note that the $\mathcal{D}$-module version of the construction of $\mathcal{F}_1^t$ is as follows. The irreducible local system $\mathcal{L}_t$ on $T \cong (\mathbb{C}^\times)^n$ corresponds to the $D$-module
    \begin{align}
        \mathbb{C}[x_1^{\pm 1}, \dots, x_n^{\pm 1}]x^{t_1} \dots x^{t_n},
    \end{align}
    where we represent $t \in T^*/Q$ as a tuple $t = (t_1, \dots, t_n)$ with $t_i \in \mathbb{R}/\mathbb{Z}$ so that each $x_i\frac{\partial}{\partial x_i}$ has eigenvalues $t_i + k$ for $k \in \mathbb{Z}$.

    Pushing forward to $\mathbb{P}^n$ and recalling that $\mathcal{D}(\mathbb{P}^n)$ is generated over $\mathcal{O}_{\mathbb{P}^1}$ by the elements $x_i\frac{\partial_j}{\partial x_j}$ for $0 \leq i, j \leq n$, so we can restrict the $\mathcal{D}$-module on $T$ described above to a module over this ring.

    Now note that performing the pullback along $p_1 : G/B \to G/P \cong \mathbb{P}^n$ is equivalent to pulling this module back along the map from $\mathcal{U}(\mathfrak{sl}_{n+1}) \to \mathcal{D}(\mathbb{P}^1)$ which sends
    \begin{align}
        E_{ij} & \mapsto x_i \frac{\partial}{\partial x_j}, & H_k \mapsto \frac{\partial}{\partial x_k} - \frac{\partial}{\partial x_{k+1}}
    \end{align}
    for any $i \neq j$ with $1 \leq i \leq j$ and any $1 \leq k \leq n-1$.

    One can then see that we have just described a $\mathcal{U}(\mathfrak{sl}_{n+1})$-module with each weight space having dimension $1$. Since every element of $\mathfrak{sl}_{n+1}$ acts by a differential operator as above, one can see that traces of elements of $\mathcal{U}(\mathfrak{sl}_{n+1})^T$ on any weight space must be polynomial in the parameter $t$. This means both conditions of Definition \ref{def:coherentfamily} are verified, and thus $\mathcal{M}_1$ is a coherent family. 
\end{proof}

\begin{lemma}\label{lem:descent}
    Let $w_0^P$ be the longest element of $\langle s_2, \dots, s_n\rangle$. Then for any $2 \leq k \leq n$, the right descent set of the element
    \[w_0^Ps_1 \cdots s_{k-1}\]
    contains all $s_j \in S$ for $j \neq k$.
\end{lemma}
\begin{proof}
    Note that the elements $w_0^Ps_1 \dots s_{k-1}$ for $2 \leq k \leq n$ can be represented as permutations, in one-line notation, as
    \begin{align*}
        (n \quad (n - 1) \quad \dots \quad (n - k + 1) \quad 1 \quad (n - k) \quad \dots \quad 2),
    \end{align*}
    which immediately gives the result.
\end{proof}

\begin{definition}
    Let $2 \leq k \leq n$. Then we define $\tilde{\mathcal{F}}_k^t \in D^b(G/B)$ by
    \begin{align*}
        \tilde{\mathcal{F}}_k^t & = \mathcal{G}_1^t *_{P} \mathrm{IC}_{s_1 \cdots s_{k-1}}^{P},
    \end{align*}
    where $\mathrm{IC}_{s_1 \cdots s_{k-1}}^{P} \in \mathrm{Perv}(P \backslash G /B)$
\end{definition}

\begin{lemma}
    The functor $-*_{P} \mathrm{IC}_{s_1 \dots s_{k-1}}^{P}$ is a direct summand of the functor $-*_{B} \mathrm{IC}_{s_1 \dots s_{k-1}}$.
\end{lemma}
\begin{proof}
    We write
    \begin{align*}
        m^P & : B\backslash G \times_P G/B \to B\backslash G/B,\\
        m^B & : B\backslash G \times_B G/B \to B\backslash G/B,\\
        a & : B\backslash G \times_B G/B \to B \backslash G\times_P G/B,\\
        p & : G/B \to G/P,
    \end{align*}
    with $\pi_1, \pi_2$ the coordinate projections from $B\backslash G \times_B G/B$ and $\overline{\pi}_1, \overline{\pi}_2$ those from $B\backslash G \times_P G/B$. Then for any $\mathcal{F} \in D^b(B\backslash G/B)$ and $\mathcal{G} \in D^b(P\backslash G /B)$,
    \begin{align*}
        (p_*\mathcal{F}) *_P \mathcal{G} & = m_*^P(\overline{\pi}_1^*p_*\mathcal{F}\otimes \overline{\pi}_2^*\mathcal{G})\\
        & = a_*\pi_1^*\mathcal{F} \otimes \overline{\pi}_2^*\mathcal{G}\\
        & = m_*^Pa_*(\pi_1^*\mathcal{F}\otimes a^*\overline{\pi}_2^*\mathcal{G})\\
        & = m^B_*(\pi_1^*\mathcal{F}\otimes \pi_2^*\mathcal{G})\\
        & = \mathcal{F} *_B \mathcal{G}
    \end{align*}
    since $\overline{\pi}_2\circ a = \pi_2$ and $m^P\circ a = m^B$. 
    
    In particular, this means $-*_{P} ((p_1)_*\mathrm{IC}_{s_1 \dots s_{k-1}}) = -*_B \mathrm{IC}_{s_1 \dots s_{k-1}}$. By the decomposition theorem, we have
    \begin{align*}
        \pi_*\mathrm{IC}_{s_1 \cdots s_{k-1}} & = \mathrm{IC}_{s_1 \cdots s_{k-1}}^{P} \oplus \mathcal{E}
    \end{align*}
    for some object $\mathcal{E} \in \mathrm{Perv}(P \backslash G/B)$, guaranteeing the statement in the lemma.
\end{proof}

\begin{corollary}\label{cor:summand}
    For any $1 \leq k \leq n$, the object $\tilde{\mathcal{F}}_k^t$ is a canonical direct summand of
    \begin{align}
        \mathcal{G}_1^t * \mathrm{IC}_{s_1 \cdots s_{k-1}},
    \end{align}
    where $\mathrm{IC}_{s_1 \cdots s_{k-1}} \in \mathrm{Perv}(B\backslash G/B)$. 
\end{corollary}
By the way this canonical summand arises (in particular, that the direct sum decomposition occurs exactly the same way independently of $t$), we then see that to show $\tilde{\mathcal{F}}_k^t$ assemble to a coherent complex, it is enough to show that the objects $\mathcal{G}_1^t * \mathrm{IC}_{s_1 \cdots s_{k-1}}$ appearing in in Corollary \ref{cor:summand} assembles to a complex family across all $t \in \mathfrak{h}^*/Q$. This follows by induction using the following lemma.
\begin{lemma}
    If $\mathcal{G}^t \in D^b(G/B)$ are objects which assemble to a coherent complex, then for any $s \in S$, the objects
    \[\mathcal{G}^t * \mathrm{IC}_s\]
    also do so.
\end{lemma}
\begin{proof}
    Let $p : G/U \to G/B$ be the projection, then note that $\mathcal{G}^t *_B \mathrm{IC}_s \cong (p^*\mathcal{G}^t) *_U \widetilde{\mathrm{IC}}_s$, where $\widetilde{\mathrm{IC}}_s$ is the corresponding object of $\mathrm{Perv}_U(G/B)$. We note that in $D^b_U(G/B)$, $\widetilde{\mathrm{IC}}_s$ can be represented as the complex of tilting objects given by
    \[\begin{tikzcd}
        \dots \arrow[r] 0 \arrow[r] & T_e \arrow[r] & T_s \arrow[r] & T_e \arrow[r] & 0 \arrow[r] & \dots,
    \end{tikzcd}\]
    with the $T_s$ term sitting in degree zero.
    Since tilting sheaves are convolution-exact, we can see that the complex $\mathcal{G}^t *_B \mathrm{IC}_s$ is then quasi-isomorphic to the total complex of
    \begin{equation}\label{eqn:doublecomplex}\begin{tikzcd}
        \dots \arrow[r] & 0 \arrow[r] & \tilde{\mathcal{G}}^t \arrow[r] & \tilde{\mathcal{G}^t} *_U T_s \arrow[r] & \tilde{\mathcal{G}}^t \arrow[r] & 0 \arrow[r] & \dots
    \end{tikzcd}\end{equation}
    We then note that the $\mathcal{U}(\mathfrak{g})$-module corresponding to $\tilde{\mathcal{G}}^t *_U T_s$ is exactly the module obtained by applying the \emph{wall-crossing functor} labelled by $s$ to the module corresponding to $\tilde{\mathcal{G}}^t$; this follows from \cite{ABG}, c.f.\ \cite{SCatO}. We note that the wall-crossing functor labeled by $s$ is a composition of \emph{translation functors}, each of which itself is a composition of tensoring with a finite-dimensional module and then projecting to the zero central character block.

    We note that each of these functors preserves the notion of a coherent complex: first, clearly tensoring with a finite-dimensional representation preserves the property of being a coherent complex. The resulting complex will be a finite direct sum of complexes with various central characters. Since then the projection to the trivial central character summand can be written as a polynomial in the generators of $\mathcal{Z}(\mathfrak{g})$, this summand must also be a coherent complex.

    Thus, since each term of (\ref{eqn:doublecomplex}) is a coherent complex, it is clear that the corresponding total complex is coherent. Since this complex is quasi-isomorphic to $\mathcal{G}^t * \mathrm{IC}_s$, the statement in the lemma must hold.
\end{proof}

In the following, let $\mu : T^*(G/B) \to \mathfrak{g}$ be the moment map whose image is the nilpotent cone, i.e.\ the Springer resolution.
\begin{lemma}\label{lem:irredcomponents}
    Suppose $e \in (\mathfrak{sl_{n+1}})^*$ is a nilpotent element belonging to the minimal nonzero nilpotent orbit (this orbit consists of elements with Jordan type given by the partition $(2, n-1)$). Then the Springer fiber $X_e = \mu^{-1}(e) \subset T^*(G/B)$ has $n$ irreducible components $X_e^k$, given by
    \begin{align}
        X_e^k &= Y_k \cap X_e,
    \end{align}
    where for any $1 \leq k \leq n$, $p_k : G/B \to G/P_{k}^{\mathrm{max}}$ is the projection, and $Y_k$ is the image in $T^*(G/B)$ of the correspondence
    \begin{equation}
        \begin{tikzcd}
            & G/P_k^{\mathrm{max}}\times_{G/B} T^*(G/B)\arrow[dl] \arrow[dr] & \\
            T^*(G/P_k^{\mathrm{max}}) & & T^*(G/B).
        \end{tikzcd}
    \end{equation}
\end{lemma}

We can then let $X_{\mathbb{O}}^k$ be the closure of $\mu^{-1}(\mathbb{O})$ where $\mathbb{O}$ is the nilpotent orbit of $e$.
\begin{lemma}\label{lem:cohsupp}
    For any $k$ with $1 \leq k \leq n$ and any $t \in \mathfrak{h}^*/Q$, the singular support of $\tilde{\mathcal{F}}_k^t$ lies in the closure of $X_\mathbb{O}^k \subset T^*(G/B)$.
\end{lemma}
\begin{proof}
    First, we claim that $\tilde{\mathcal{F}}_k^t$ lies in the image of $p_k^*$; this implies that the singular support of $\tilde{\mathcal{F}}_k$ lies in $Y_k$. To do so, note that by its definition as a convolution, it is enough to check that $\mathrm{IC}_{s_1 \cdots s_{k-1}}^{P}$ lies in the image of $p_k^*$. Note that $P_k^\mathrm{max}$ corresponds to the parabolic subgroup $\langle s_1, \dots, s_{k-1}, s_{k+1}, \dots, s_n\rangle$, and so $\mathrm{IC}_{s_1 \cdots s_{k}}^{P}$ lying in the image of $p_k^*$ is exactly the statement in Lemma \ref{lem:descent} that the right descent set of $w_0^Ps_1 \dots s_{k-1}$ contains every simple reflection in this parabolic subgroup of $W$.
    
    By Lemma \ref{lem:irredcomponents}, to show that it is contained in $X_{\mathbb{O}}^k$, it remains only to show that it is also contained in $X_{\mathbb{O}}$. In \cite[Proposition 1.6]{Mathieu} (c.f.\ \cite{Fernando}), it is explained that the associated variety of any coherent family lies inside the minimal nonzero nilpotent orbit; this is exactly the statement that the singular support of $\tilde{\mathcal{F}}_k^t$ lies in $X_{\mathbb{O}}$.
\end{proof}

For the following inductive argument, note that in the present case, we have
\begin{align*}
    \mathrm{IC}_{s_1} * \dots * \mathrm{IC}_{s_{k-1}} \cong \mathrm{IC}_{s_1\cdots s_{\mathrm{k-1}}}.
\end{align*}
\begin{corollary}\label{cor:pervandconst}
    For any $1 \leq k \leq n$, the object $\tilde{\mathcal{F}}_k^t$ is the direct sum of a perverse sheaf $\mathcal{P}_k^t$ with some number of cohomologically-shifted copies of the constant sheaf.
\end{corollary}
\begin{proof}
    By induction, the description of $\tilde{\mathcal{F}}_k^t$ in Corollary \ref{cor:summand} shows that it suffices to show that if this statement holds for $\tilde{\mathcal{F}}_k^t$, then it also holds for $\tilde{\mathcal{F}}_k^t*\mathrm{IC}_{s_k}$.

    First note that
    \begin{align*}
        \tilde{\mathcal{F}}_k^t * \mathrm{IC}_{s_k} & \cong p_{s_k}^*p_{s_k*}\tilde{\mathcal{F}}_k^t[1], & p_{s_k} & : G/B \to G/P_{s_k}.
    \end{align*}
    for $P_{s_k}$ the \emph{minimal} parabolic corresponding to the simple reflection $s_k$. Since $p_{s_k}$ is smooth, it is enough to show that $p_{s_k*}\tilde{\mathcal{F}}_k^t[1]$ is perverse (up to some constant sheaf factors). To do so, we use the microlocalization functors discussed in \cite[\S 4.3]{KS}. We note that these functors are compatible with pushforward as detailed in \cite[4.3.4]{KS}: if $Z \subset G/B$ is a subvariety, and we consider the maps $f^{P_{s_k}}$ and $f^B$ as in
    \begin{equation}
        \begin{tikzcd}
            & G/P_{s_k}\times_{G/B} T^*(G/B)\arrow[dl, "f^{P_{s_k}}"'] \arrow[dr, "f^B"] & \\
            T^*(G/P_{s_k}) & & T^*(G/B),
        \end{tikzcd}
    \end{equation}
    then $\mu_{p_{s_k}(Z)}(p_{s_k*}\tilde{\mathcal{F}}_k^t) = f^{P_{s_k}}_*f^{B*}\mu_Z(\tilde{\mathcal{F}}_k^t)$.

    By Lemma \ref{lem:cohsupp}, $\tilde{\mathcal{F}}_k$ has singular support in the closure of $X_{\mathbb{O}}^k \subset T^*(G/B)$. Note that we can write $X_{\mathbb{O}}^k$ as the union of some conormals $T_{Z_i}^*(G/B)$; let $Z$ be any such $Z_i$ appearing in this union. 

    We now note that among the $n$ components $\{X_{\mathbb{O}}^i\}_{i=1}^n$, $X_{\mathbb{O}}^k$ is the unique component for which the map $f^{P_{s_k}}$ is generically one-to-one on $(f^B)^{-1}(X_{\mathbb{O}}^k)$. This is because $X_{\mathbb{O}}^k$ is the image of a correspondence defined by pullback from the maximal parabolic $G/P_k^{\mathrm{max}}$, and $P_k^{\mathrm{max}}$ is the parabolic corresponding to $$\langle s_1, \dots, s_{k-1}, s_{k+1}, \dots, s_n\rangle.$$ Thus this claim follows from the explicit linear-algebraic description of the components $X_{\mathbb{O}}^k$ in terms of flags given, for example, in \cite[\S II.5]{Spalt}. This perspective also shows that $(f^B)^{-1}(Z)$ can be thought of as an intersection of $Z$ with $G/P_{s_k} \times_{G/B} T^*(G/B)$ occuring in codimension $1$ at the generic locus. Together, these facts imply that for any $Z$ as chosen above, $f_{*}^{P_{s_k}}f^{B*}\mu_Z(\tilde{\mathcal{F}}_k^t)[1] = \mu_{p_{s_k}(Z)}(p_{s_k*}\tilde{\mathcal{F}}_k^t[1])$ is perverse.

    Now note that if this were true for all choices of $Z$ with $T_Z^*(G/B)$ lying in the singular support of $\tilde{\mathcal{F}}^t_k$, this would show that $p_{s_k*}\tilde{\mathcal{F}}_k^t[1]$ is itself perverse. There remains, though, the choice of $Z = G/B$, giving rise to the zero section of $T^*(G/B)$. Thus, this argument shows not that $p_{s_k*}\tilde{\mathcal{F}}_k^t[1]$ is perverse, but rather that its perverse cohomology in any degree other than $0$ must have singular support on the zero section of $T^*(G/B)$. The only objects with this singular support are constant sheaves, and thus the claim that $p_{s_k*}\tilde{\mathcal{F}}_k^t[1]$ is perverse up to some constant sheaf factors indeed holds, proving the corollary.
\end{proof}

Since the objects $\tilde{\mathcal{F}}_k^t$ already assemble to a coherent complex, this means that after adding or removing some number of copies of the (unshifted) constant sheaf from each $\mathcal{P}_k^t$ (as was defined in Corollary \ref{cor:pervandconst}), the resulting sheaves assemble to a coherent family. To be more precise about how this new object $\mathcal{F}_k^t$ is constructed in terms of $\mathcal{P}_k^t$, we simply add or remove some number of constant sheaves from $\mathcal{P}_k^t$ so that the trace function in $\lambda$ of any element in $\mathcal{U}(\mathfrak{g})^T$ on its corresponding weight module is equal to that obtained from $\tilde{\mathcal{F}}_k^t$.
\begin{definition}
    Let $\mathcal{F}_k^t$ be the unique coherent family of sheaves obtained from $\mathcal{P}_k^t$ by adding or removing, by direct sum, some number of copies of the constant sheaf.

    Let $\mathcal{M}_k[t]$ be the module corresponding to each $\mathcal{F}_k^t$, so that
    \[\mathcal{M}_k = \oplus_{t \in \mathfrak{h}^*/Q} \mathcal{M}_k[t]\]
    is a coherent family.
\end{definition}

Since Lemma \ref{lem:cohsupp} guarantees that the coherent families $\{\mathcal{M}_k\}_{k=1}^n$ constructed in Section \ref{sec:geomconstruction} are distinct, we obtain the following. Note that for the classification we consider coherent families only up to their Jordan--H{\"o}lder filtration, as is done in \cite{Mathieu}.
\begin{theorem}
    The collection $\{\mathcal{M}_k\}_{k=1}^n$ of coherent families is the complete set of isomorphism classes of the coherent families for $\mathfrak{sl}_{n+1}$ described in Theorem \ref{thm:mathieucoherent}.
\end{theorem}

\subsection{A universal description}
Recall now the ``ind-universal local system" $\tilde{\mathcal{L}}$ introduced in Definition \ref{def:locsys}.

\begin{definition}
    For any $k$ with $1 \leq k \leq n$, let $\widetilde{\mathcal{M}}_k$ be the coherent family obtained exactly as $\mathcal{M}_k[t]$ is defined in Section \ref{sec:geomconstruction}, but with the role of the local system $\mathcal{L}_t$ replaced by the ind-universal local system $\tilde{\mathcal{L}}$.
\end{definition}

Since for any irreducible local system $\mathcal{L}_t$, there is a natural inclusion $\mathcal{L}_t \hookrightarrow \tilde{\mathcal{L}}$, we see that for any $t \in T^*/Q$, there is a natural inclusion
\begin{align}
    \mathcal{M}_k[t] \hookrightarrow \widetilde{\mathcal{M}}_k,
\end{align}
and therefore a natural inclusion
\begin{align*}
    \mathcal{M}_k \hookrightarrow \widetilde{\mathcal{M}}_k.
\end{align*}
We can then view $\widetilde{\mathcal{M}}_k$ as a natural universal analog of a coherent family which contains as a submodule the irreducible coherent family constructed in Section \ref{sec:geomconstruction}.

For example, when $\mathfrak{g} = \mathfrak{sl}_2$, there is a unique semisimple coherent family $\mathcal{M}$; we let $\widetilde{\mathcal{M}}$ be its universal analog described at present. Then in this case, $\widetilde{\mathcal{M}}[0]$ is the module pictured in Figure \ref{fig:m0}, while $\mathcal{M}[0]$ is the module pictured in Figure \ref{fig:m0simple}.

\begin{figure}[htb]
    \centering
    \begin{tikzpicture}[scale=0.8, every node/.style={circle, fill=black, inner sep=1.5pt}, >=stealth]
            \foreach \x in {-6, -4,-2,2,4, 6}
            \node (lw\x) at (\x,-3) {};
            \node (llw0) at (0,-3) {};
            \foreach \x in {-4,-2,0,2,4}
            \node[fill=none] at (\x,-3.7) {$\x$};
            \node[fill=none] at (-6.65,-3) {$\dots$};
            \node[fill=none] at (6.65,-3) {$\dots$};

            \draw[->, thick, black, bend left=30] (lw-2) to node[above, sloped, midway, fill=none, inner sep=0pt, outer sep=0pt] {} (lw-4);
            \draw[->, thick, black, bend right=-30] (lw-4) to node[above, sloped, midway, fill=none, inner sep=0pt, outer sep=0pt] {} (lw-2);
            \draw[->, thick, black, bend right=-30] (lw-4) to node[above, sloped, midway, fill=none, inner sep=0pt, outer sep=0pt] {} (lw-6);
            \draw[->, thick, black, bend right=-30] (lw-6) to node[above, sloped, midway, fill=none, inner sep=0pt, outer sep=0pt] {} (lw-4);
            
            \draw[->, thick, black, bend right=-30] (lw2) to node[above, sloped, midway, fill=none, inner sep=0pt, outer sep=0pt] {} (lw4);
            \draw[->, thick, black, bend left=30] (lw4) to node[above, sloped, midway, fill=none, inner sep=0pt, outer sep=0pt] {} (lw2);
            \draw[->, thick, black, bend left=30] (lw4) to node[above, sloped, midway, fill=none, inner sep=0pt, outer sep=0pt] {} (lw6);
            \draw[->, thick, black, bend left=30] (lw6) to node[above, sloped, midway, fill=none, inner sep=0pt, outer sep=0pt] {} (lw4);

            \draw[->, thick, black, bend left=0] (lw2) to node[above, sloped, midway, fill=none, inner sep=0pt, outer sep=0pt] {} (llw0);
            \draw[->, thick, black, bend left=0] (lw-2) to node[above, sloped, midway, fill=none, inner sep=0pt, outer sep=0pt] {} (llw0);
        \end{tikzpicture}
    \caption{A schematic picture of the module $\mathcal{M}[0]$.}
    \label{fig:m0simple}
\end{figure}
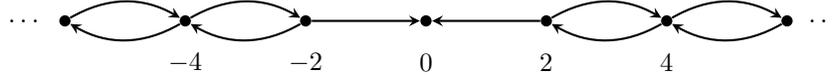

\clearpage

\begingroup
\bibliographystyle{alphaurl}
\bibliography{bibl}

\begin{thebibliography}{BBAMY}

\bibitem[ABG04]{ABG}
Sergey Arkhipov, Roman Bezrukavnikov, and Victor Ginzburg.
\newblock Quantum groups, the loop {G}rassmannian, and the {S}pringer resolution.
\newblock {\em J. Amer. Math. Soc.}, 17(3):595--678, 2004.
\newblock \href {https://doi.org/10.1090/S0894-0347-04-00454-0} {\path{doi:10.1090/S0894-0347-04-00454-0}}.

\bibitem[Ark22]{Arkor}
Nathanael Arkor.
\newblock {\em Monadic and Higher-Order Structure}.
\newblock PhD thesis, Apollo - University of Cambridge Repository, 2022.
\newblock URL: \url{https://www.repository.cam.ac.uk/handle/1810/338940}, \href {https://doi.org/10.17863/CAM.86347} {\path{doi:10.17863/CAM.86347}}.

\bibitem[BB81]{BBloc}
Alexandre Beilinson and Joseph Bernstein.
\newblock Localisation de {$\mathfrak{g}$}-modules.
\newblock {\em C. R. Acad. Sci. Paris S\'er. I Math.}, 292(1):15--18, 1981.

\bibitem[BBAMY]{BBAMY}
Roman Bezrukavnikov, Pablo Boixeda~Alvarez, Michael McBreen, and Zhiwei Yun.
\newblock Affine {S}pringer fibers and the small quantum group.
\newblock To appear.

\bibitem[BBP02]{BBP}
Roman Bezrukavnikov, Alexander Braverman, and Leonid Positselskii.
\newblock Gluing of abelian categories and differential operators on the basic affine space.
\newblock {\em J. Inst. Math. Jussieu}, 1(4):543--557, 2002.

\bibitem[BGG76]{BGG}
I.~N. Bernstein, I.~M. Gelfand, and S.~I. Gelfand.
\newblock A certain category of {${\mathfrak{g}}$}-modules.
\newblock {\em Funkcional. Anal. i Prilozen.}, 10(2):1--8, 1976.

\bibitem[BGS96]{BGS}
Alexander Beilinson, Victor Ginzburg, and Wolfgang Soergel.
\newblock Koszul duality patterns in representation theory.
\newblock {\em J. Amer. Math. Soc.}, 9(2):473--527, 1996.
\newblock \href {https://doi.org/10.1090/S0894-0347-96-00192-0} {\path{doi:10.1090/S0894-0347-96-00192-0}}.

\bibitem[BY13]{BY}
Roman Bezrukavnikov and Zhiwei Yun.
\newblock On {K}oszul duality for {K}ac-{M}oody groups.
\newblock {\em Represent. Theory}, 17:1--98, 2013.

\bibitem[DT25]{DhTa}
Gurbir Dhillon and Jeremy Taylor.
\newblock The universal monodromic arkhipov--bezrukavnikov equivalence, 2025.
\newblock URL: \url{https://arxiv.org/abs/2501.14156}, \href {https://arxiv.org/abs/2501.14156} {\path{arXiv:2501.14156}}.

\bibitem[Fer90]{Fernando}
S.~L. Fernando.
\newblock Lie algebra modules with finite-dimensional weight spaces. {I}.
\newblock {\em Trans. Amer. Math. Soc.}, 322(2):757--781, 1990.
\newblock \href {https://doi.org/10.2307/2001724} {\path{doi:10.2307/2001724}}.

\bibitem[KL79]{KL1}
David Kazhdan and George Lusztig.
\newblock Representations of {C}oxeter groups and {H}ecke algebras.
\newblock {\em Invent. Math.}, 53(2):165--184, 1979.

\bibitem[KL88]{KL}
D.~Kazhdan and G.~Laumon.
\newblock Gluing of perverse sheaves and discrete series representation.
\newblock {\em J. Geom. Phys.}, 5(1):63--120, 1988.

\bibitem[KS90]{KS}
Masaki Kashiwara and Pierre Schapira.
\newblock {\em Sheaves on manifolds}, volume 292 of {\em Grundlehren der mathematischen Wissenschaften [Fundamental Principles of Mathematical Sciences]}.
\newblock Springer-Verlag, Berlin, 1990.
\newblock With a chapter in French by Christian Houzel.
\newblock \href {https://doi.org/10.1007/978-3-662-02661-8} {\path{doi:10.1007/978-3-662-02661-8}}.

\bibitem[Lus19]{LusRemarks}
George Lusztig.
\newblock Remarks on affine {S}pringer fibres.
\newblock {\em Bulletin of the Institute of Mathematics Academia Sinica NEW SERIES}, 2019.
\newblock URL: \url{https://api.semanticscholar.org/CorpusID:119260683}.

\bibitem[Mat00]{Mathieu}
Olivier Mathieu.
\newblock Classification of irreducible weight modules.
\newblock {\em Ann. Inst. Fourier (Grenoble)}, 50(2):537--592, 2000.
\newblock \href {https://doi.org/10.5802/aif.1765} {\path{doi:10.5802/aif.1765}}.

\bibitem[MF23]{CMF3}
Calder Morton-Ferguson.
\newblock {P}olishchuk's conjecture and {K}azhdan--{L}aumon representations, 2023.
\newblock \href {https://arxiv.org/abs/2309.13462} {\path{arXiv:2309.13462}}.

\bibitem[MF24]{217}
Calder Morton-Ferguson.
\newblock Symplectic {F}ourier-{D}eligne transforms on {$G/U$} and the algebra of braids and ties.
\newblock {\em Int. Math. Res. Not. IMRN}, 2024(13):10219--10235, 2024.
\newblock \href {https://doi.org/10.1093/imrn/rnae052} {\path{doi:10.1093/imrn/rnae052}}.

\bibitem[MF25]{KLCatO}
Calder Morton-Ferguson.
\newblock {K}azhdan--{L}aumon category {$\mathcal{O}$}, {B}raverman-{K}azhdan {S}chwartz space, and the semi-infinite flag variety.
\newblock {\em Represent. Theory}, 29:1--34, 2025.
\newblock \href {https://doi.org/10.1090/ert/684} {\path{doi:10.1090/ert/684}}.

\bibitem[MS07]{MS}
Volodymyr Mazorchuk and Catharina Stroppel.
\newblock On functors associated to a simple root.
\newblock {\em J. Algebra}, 314(1):97--128, 2007.
\newblock \href {https://doi.org/10.1016/j.jalgebra.2007.03.015} {\path{doi:10.1016/j.jalgebra.2007.03.015}}.

\bibitem[Pol01]{P}
Alexander Polishchuk.
\newblock Gluing of perverse sheaves on the basic affine space.
\newblock {\em Selecta Math. (N.S.)}, 7(1):83--147, 2001.
\newblock With an appendix by R. Bezrukavnikov and the author.

\bibitem[Qui67]{Quil}
Daniel~G. Quillen.
\newblock {\em Homotopical algebra}, volume No. 43 of {\em Lecture Notes in Mathematics}.
\newblock Springer-Verlag, Berlin-New York, 1967.

\bibitem[Soe90]{SCatO}
Wolfgang Soergel.
\newblock Kategorie {$\mathcal{O}$}, perverse {G}arben und {M}oduln \"{u}ber den {K}oinvarianten zur {W}eylgruppe.
\newblock {\em J. Amer. Math. Soc.}, 3(2):421--445, 1990.

\bibitem[Spa82]{Spalt}
Nicolas Spaltenstein.
\newblock {\em Classes unipotentes et sous-groupes de {B}orel}, volume 946 of {\em Lecture Notes in Mathematics}.
\newblock Springer-Verlag, Berlin-New York, 1982.

\bibitem[Ulm68]{Ulmer}
Friedrich Ulmer.
\newblock Properties of dense and relative adjoint functors.
\newblock {\em J. Algebra}, 8:77--95, 1968.
\newblock \href {https://doi.org/10.1016/0021-8693(68)90036-7} {\path{doi:10.1016/0021-8693(68)90036-7}}.

\bibitem[Yan21]{Monadic}
Lior Yanovski.
\newblock The monadic tower for $\infty$-categories, 2021.
\newblock \href {https://arxiv.org/abs/2104.01816} {\path{arXiv:2104.01816}}.

\bibitem[Zuc77]{Zuckerman}
Gregg Zuckerman.
\newblock Tensor products of finite and infinite dimensional representations of semisimple {L}ie groups.
\newblock {\em Ann. of Math. (2)}, 106(2):295--308, 1977.
\newblock \href {https://doi.org/10.2307/1971097} {\path{doi:10.2307/1971097}}.

\end{thebibliography}
\endgroup

\end{document}